\RequirePackage{fix-cm}
\documentclass[smallextended]{svjour3}       
\smartqed  

\usepackage{subfig}

\usepackage{graphicx}
\usepackage{appendix}

\newtheorem{rmk}{Remark}

\usepackage{tikz-cd}
\usepackage{amsfonts}      
\usepackage{amsmath}
\usepackage{amssymb}
\usepackage[margin=.8in]{geometry} 
\usepackage{parskip}
\usepackage{hyperref}
\usepackage{algorithm}
\usepackage{algorithm}
\usepackage[noend]{algorithmic}
\usepackage{framed}
\usepackage{mathtools}
\usepackage{bbm}

\newcommand{\su}{\subseteq}

\newcommand{\ImKer}{\mathrm{ImKer}}
\newcommand{\Coker}{\mathrm{Coker}}
\newcommand{\im}{\mathrm{Im}}
\newcommand{\Ker}{\mathrm{Ker}}
\newcommand{\Z}{\mathbb{Z}}
\newcommand{\Q}{\mathbb{Q}}
\newcommand{\R}{\mathbb{R}}

\newcommand{\C}{\mathbb{C}}
\newcommand{\colspace}{\textsc{Colspace}}

\newcommand{\K}{\mathcal{K}}

\newcommand{\diag}{\mathrm{diag}}
\newcommand{\id}{\mathrm{id}}

\newcommand{\Col}{\mathrm{Col}}
\newcommand{\low}{\mathrm{low}}

\newcommand{\Chains}{C}
\newcommand{\spann}{\mathrm{span}}

\newcommand{\Ps}{\mathcal{P}}

\newcommand{\Ralt}{\widehat{R}}
\newcommand{\Valt}{\widehat{V}}
\newcommand{\row}{\mathrm{Row}}
\newcommand{\lead}{\mathrm{lead}}

\newcommand{\communicated}[1]{\vspace{0.5em}\noindent\textit{Communicated by #1}\vspace{1em}}

\usepackage{cite}   

\begin{document}

\title{Interval Decomposition of Persistence Modules over a Principal Ideal Domain\thanks{G. H.-P. acknowledges the support of the Centre for Topological Data Analysis of Oxford, the Mathematical Institute of Oxford, EPSRC grant EP/R018472/1, and NSF grant number 1854748. J. L. acknowledges funding from NSF grant number 1829071, NSF grant number 2136090, and NSF grant number 1922952 through the Algorithms for Threat Detection (ATD) program. }
}

\titlerunning{Interval Decomposition over a PID}        

\author{Jiajie Luo$^*$\thanks{$^*$Corresponding Author}         \and
        Gregory Henselman-Petrusek 
}


\institute{Jiajie Luo \at
              Knowledge Lab \\ University of Chicago \\
              Chicago, Illinois, United States of America\\
              \email{jerryluo8@uchicago.edu}      
           \and
           Gregory Henselman-Petrusek \at
            Pacific Northwest National Laboratory \\
            Richland, Washington, United States of America\\
            \email{gregory.roek@pnnl.gov}
}

\date{Received: date / Accepted: date}

\maketitle

\communicated{Peter Bubenik}

\begin{abstract}
The study of persistent homology has contributed new insights and perspectives into a variety of interesting problems in science and engineering.
Work in this domain relies on the result that any finitely-indexed persistence module of finite-dimensional vector spaces
admits an interval decomposition --- that is, a decomposition as a direct sum of simpler components called interval modules.
This result fails if we replace vector spaces with modules over more general coefficient rings.
To address this problem, we introduce an algorithm to determine whether or not a persistence module of pointwise free and finitely-generated modules over a principal ideal domain (PID) splits as a direct sum of interval submodules. 
If one exists, our algorithm outputs an interval decomposition.
When considering persistence modules with coefficients in $\Z$ or $\Q[x]$, our algorithm computes an interval decomposition in polynomial time.
This is the first algorithm with these properties of which we are aware. 
We also show that a persistence module of pointwise free and finitely-generated modules over a PID splits as a direct sum of interval submodules if and only if the cokernel of every structure map is free. 
This result underpins the formulation of our algorithm. 
It also complements prior findings by Obayashi and Yoshiwaki regarding persistent homology, including a criterion for field independence and an algorithm to decompose persistence homology modules.

\end{abstract}

%

%
\section{Introduction}
\label{sec:intro}

A (finitely-indexed) \emph{persistence module} is a functor from a (finite) totally-ordered poset category to a category of modules over a commutative ring $R$.
While abstract in definition, persistence modules underlie \emph{persistent homology} (PH) \cite{edelsbrunner_topological_2002, ghrist_barcodes_2007, dey_computational_2022}, an area of topological data analysis (TDA) that has contributed to applied and theoretical scientific research in domains including medicine \cite{qaiser_fast_2019}, neuroscience \cite{curto_what_2017, sizemore_cliques_2018, stolz_topological_2021}, image processing \cite{diaz-garcia_image-based_2018, perea_klein-bottle-based_2014}, material science \cite{onodera_understanding_2019}, genomics \cite{rabadan_topological_2019}, and signal processing \cite{robinson2014topological}.

PH aims to capture and quantify topological features (e.g., holes) in data. 
In PH, one obtains a persistence module by taking the homology of a nested sequence of topological spaces. 
More generally, one can obtain a persistence module by taking the homology of a sequence of topological spaces with continuous maps between them.

Under certain conditions, a persistence module admits an interval decomposition, in which it decomposes into indecomposable pieces called \emph{interval modules}. 
Interval decompositions are rich in information. 
For example, they provide combinatorial invariants called \emph{persistence diagrams}, which can be used to vectorize topological features, as well as cycle representatives, which aid in interpreting the components of a persistence diagram \cite{robins_computational_2000,edelsbrunner_topological_2002,zomorodian_computing_2005}.

The literature on TDA has focused traditionally on homology groups with coefficients in a field because persistence modules with field coefficients are guaranteed to decompose into interval modules (see Theorem \ref{thm:gabriel}).
However, there is an increasing amount of research on PH with different coefficients (e.g., the ring of integers) \cite{patel_generalized_2018, gulen_galois_2023,patel_mobius_2023}. 
A fundamental question that one can ask is how to determine whether or not a persistence module decomposes into interval modules.

We answer this question for any persistence module $f$ that is \emph{pointwise free and finitely-generated} over a principal ideal domain (PID) $R$. 
Concretely, this condition means that $f:\{0, \ldots, m\} \to R\textrm{-Mod}$ is a functor in which $\{0, \ldots, m\}$ is the poset category with the usual order on the integers $0, \ldots, m$ and $f_a$ is a free and finitely-generated $R$-module for all $a\in \{0, \ldots, m\}$. 
In this setting, we address the following problems.

\begin{problem}\label{prob:splits}
    Determine whether $f$ splits as a direct sum of interval submodules.
\end{problem}

\begin{problem}\label{prob:int_decomp}
    If an interval decomposition does exist, then compute one explicitly.
\end{problem}

For any persistence module $f$ that is pointwise free and finitely-generated, we have the following theorems that address Problems \ref{prob:splits} and \ref{prob:int_decomp}.

\begin{theorem}[Main theoretical result]
\label{thm:problem_statement}
Let $f:\{0, \ldots, m\} \to R\textrm{-Mod}$ be a persistence module that is pointwise free and finitely-generated over $R$.
The persistence module $f$ splits into a direct sum of interval modules if and only if the cokernel of every homomorphism $f(a\leq b):f_a\to f_b$ is free.
\end{theorem}

Theorem \ref{thm:problem_statement} helps inform Theorem \ref{thm:algorithmic_problem}, which is our main computational result.

\begin{theorem}[Main computational result]
\label{thm:algorithmic_problem}
Let $f:\{0, \ldots, m\} \to R\textrm{-Mod}$ be a persistence module that is pointwise free and finitely-generated over $R$.
There is a formal procedure (see Algorithm \ref{alg:matrix}) to determine whether or not $f$ admits an interval decomposition and, if so, to explicitly construct such a decomposition. The procedure is finite (respectively, polynomial) time if matrices over the coefficient ring $R$ can be multiplied and placed in Smith normal form in finite (respectively, polynomial)\footnote{These procedures are known to be polynomial time when $R$ is $\Z$ or $\Q[x]$. See Remark \ref{rmk:polynomial-time_Z}.} time.
\end{theorem}

On the one hand, these results give theoretical and algorithmic insights into a class of problems --- involving the structural properties of persistence modules --- which is fundamental in character, and about which little is known.
Nevertheless, these problems have diverse connections within mathematics. 
For example, Mischaikow and Weibel \cite{conley_integers} related the Conley index \cite{Mischaikow1999TheCI}, an invariant for studying dynamical systems, to persistence modules\footnote{Mischaikow and Weibel studied the Conley index over integers by reinterpreting it as a $\Z[t]$-module. Similarly, persistence modules with integer coefficients can also be viewed as $\Z[t]$-modules.}; they noted that the Conley index had stronger distinguishing capabilities when using integer coefficients than when using coefficients in a field.
Additionally, the ability to reason about modules with PID coefficients (and in particular, $\Z$-coefficients) is important to several forms of dimension reduction involving circular and projective coordinates (see \cite{perea_multiscale_2018,perea_sparse_2020,de_silva_persistent_2009,scoccola_toroidal_2023}). 
More broadly, the persistence module of a filtration topological spaces (or other data) with PID coefficients can have richer algebraic structure than using field coefficients, and thus yield additional information.\footnote{Integer homology carries maximal information, as formalized by the Universal Coefficient Theorem \cite{hatcher}. Moreover, a majority of approaches to understanding interval decomposition of persistence modules with field coefficients rely on structure which does not generalize to PID coefficients, such as locality of the associated endomorphism ring \cite{botnan2020decomposition}.}
Understanding this richer structure has spurred several branches of active research that generalize persistence theory
\cite{patel_generalized_2018,patel_mobius_2023,botnan_introduction_2023,carlsson_zigzag_2010,carlsson_theory_2009}, which have natural synergy with the methods proposed here. 
For example, in the case where persistence modules (with integer coefficients) do decompose into intervals, the associated generalized persistence diagram \cite{patel_generalized_2018} can be computed with high efficiency. 

On the other hand, our results also address knowledge gaps in areas of TDA that are already well studied.
One example is the so-called field-choice problem, which asks, given a sequence of topological spaces and continuous maps $X_1 \to \cdots \to X_N$, whether the associated  persistence diagram is invariant to the choice of coefficient field in homology (see Section \ref{sec:motivation}). 
A more sophisticated variant of this problem asks whether it is possible to obtain a basis of integer cycle representatives which is universal with respect to coefficients (see Section \ref{sec:universal_cycle_reps}).
That is, for each point in a persistence diagram, we seek a cycle representative $z$ for the corresponding interval module in integer persistent homology, such that $z \otimes K$ is a cycle representative for the corresponding interval module in persistent homology with $K$ coefficients, for any coefficient field $K$.

\noindent\fbox{%
  \parbox{\textwidth}{%
\textbf{Where the main problems in PH have or have not been solved (simplex-wise filtrations)}
\\

Obayashi and Yoshiwaki provided algorithmic solutions to the problems described above, in the setting of simplex-wise (or, more generally, cell-wise) filtrations \cite{obayashi2020field,obayashi_field_2023}. These algorithms reduce to applying the standard matrix reduction algorithm for computing persistent homology to the boundary matrix $D$, yielding a so-called $R=DV$ decomposition (see Appendix \ref{sec:standardalgorithm}). The authors show that, when a filtration is simplex-wise, the associated persistence module splits into interval submodules over $\Z$ if and only if the trailing nonzero entry in each nonzero column of $R$ equals $1$ or $-1$. (Recall that the trailing nonzero entry in a vector $v = (v_1, \ldots, v_m)$ is the last entry $v_i$ which is nonzero.)
\\

However, many filtrations in TDA are not simplex-wise (or cell-wise); in these cases, the method of \cite{obayashi2020field,obayashi_field_2023} does not apply.

\begin{example}[Failure of $R=DV$ test, for non-cell-wise filtrations]
    \label{ex:simplexwisefailure}
    Let $C$ be a chain complex such that $C_0 = \Z$, $C_1 = \Z^2$, and $C_i = 0$ for $i \notin \{0,1\}$. Suppose that the boundary matrix $\partial: C_1 \to C_0$ is represented by the matrix $\begin{bmatrix}
        2 & 3
    \end{bmatrix}$, and let $\mathcal{K} = \{\mathcal{K}_0, \mathcal{K}_1\}$ be the filtration of chain complexes given by $\mathcal{K}_0 = 0$ and $\mathcal{K}_1 = C$. 
    Then the matrix $R$ produced by the standard algorithm will have a non-zero column with trailing coefficient equal to $2$, but the persistence module $H_k(\mathcal{K}; \Z)$ splits as a direct sum of interval modules of free abelian groups for all $k$.
\end{example}

\begin{example}[Failure of simplex-wise refinement]
    The Vietoris--Rips complex --- one of the most common models used to produce a filtered simplicial complex from point-cloud data --- fails the simplex-wise condition because multiple simplices can enter the filtration on the addition of a single edge. 
    This challenge is typically addressed by refining the filtration with a linear order on simplices (e.g., breaking ties for equal-diameter simplices with lexicographic order), but doing so introduces an artificial source of variance which can obscure the underlying data. 

In an extreme case, one might sample a point cloud where all points are equidistant; the corresponding Vietoris--Rips filtration $X: X_1 \subseteq X_2$ would consist of nested spaces, where $X_1$ is a disconnected set of vertices and $X_2$ is the complete simplex on that vertex set. 
By refining this object to a simplex-wise filtration $Y: Y_0 \subseteq \cdots \subseteq Y_N$, we can easily generate intermediate spaces with torsion (provided enough vertices are present), such as copies of the Klein bottle and $\mathbb{R}\mathbf{P}^2$. In this case, any algorithm that determines field-choice independence of persistence diagrams for simplex-wise filtrations will correctly report that $Y$ does depend on the choice of field coefficient. 
However, the persistence diagram of $X$ is clearly field-independent. 
 
 \emph{Aside:} One need not introduce torsion to break the decomposition; Section~\ref{sec:motivation}, for example, provides a simple example of a filtration in which all integer homology groups are free, and the corresponding persistence module fails to decompose nevertheless.

End-users of persistent homology software might reasonably ask how often simplex-wise refinement actually does lead to false negatives, in practice. 
To the best of our knowledge, this question remains unanswered. 
Given the complex nature of the problem, numerical experiments on a wide range of data types would likely be needed to  obtain practical answers for the end-user community as a whole.
Enlarging the foundation of available algorithms to conduct these experiments is one of the main objectives of the present work. 
Moreover, in those cases where refinement does break the decomposability of a filtration, we provide  a (unique, to the best of our knowledge) method (by extending Algorithm \ref{alg:matrix}, see Section \ref{sec:universal_cycle_reps}) to extract a basis of cycle representatives which is compatible with every choice of coefficient field. 
\end{example}
  }%
}

\subsection{Contributions}

We show that a persistence module of pointwise free and finitely-generated modules over a PID splits as a direct sum of interval modules if and only if every structure map has a free cokernel.
We then provide an algorithm that either (a) computes such a decomposition explicitly or (b) verifies that no such decomposition exists.  

\subsection{Organization}

This paper is organized as follows. We provide the necessary background on persistence modules and PH in Section \ref{sec:background}. 
We present motivating applications of our work in Section \ref{sec:applications}.
We present and briefly prove the uniqueness of interval decompositions, when they exist, in Section \ref{sec:uniqueness_of_interval_decompositions}.
We present the necessary and sufficient conditions for the existence of interval decompositions for pointwise free and finitely-generated persistence modules over PIDs in Section \ref{sec:interval_decomp}, and we prove necessity and sufficiency in Sections \ref{sec:interval_decomp} and \ref{sec:interval_decomp_hard}, respectively. 
We present an algorithm (Algorithm \ref{alg:matrix}) to compute an interval decomposition in the language of matrix algebra in Section \ref{sec:matrix_algorithm}; this algorithm is based on Algorithm \ref{alg:simple}, which we formulated to prove sufficiency in Section \ref{sec:interval_decomp_hard}. 
We discuss and summarize the time complexity of our algorithm, along with related algorithms, in Section \ref{sec:complexity_Summary}.
We conclude and discuss our results in Section \ref{sec:conclusion}.

In Appendix \ref{sec:interval_decomp_hard_old}, we present an alternate proof of the sufficiency that is presented in Section \ref{sec:interval_decomp_hard}.
We provide information on Smith normal form (SNF) and other relevant matrix-algebra facts in Appendix \ref{appendix:smith_normal_form}. 
We provide a discussion on the standard algorithm for computing persistent homology in Appendix \ref{sec:standardalgorithm}; this discussion includes some new extensions of known results concerning algorithmic decomposition of persistent homology modules. 

\section{Background}
\label{sec:background}

This section discusses the necessary background for persistence modules. See \cite{oudot_persistence_2015,scheck_2022_book} for a more thorough discussion on persistence modules and its relation to PH.

\subsection{Persistence modules}
\label{sec:persistence_modules}

Let $R$ be a ring. A (finitely-indexed) \emph{persistence module} with $R$-coefficients is a functor $f: \{0,1,\ldots,m\}\to R\textrm{-Mod}$, where $\{0,1,\ldots,m\}$ is a finite totally-ordered poset category and $R\textrm{-Mod}$ is the category of $R$-modules. 
That is, $f$ can be viewed as an algebraic structure that consists of (1) a finite collection $\{f_i\}_{i=0}^m$ of $R$-modules and (2) $R$-homomorphisms $\{f(i\leq j):f_i\to f_j\}_{0\leq i\leq j\leq m}$ (which we call \emph{structure maps}) between the aforementioned $R$-modules such that
$f(i\leq i)=\textrm{id}_{f_i}$ and $f(i\leq j)=f(s\leq j)\circ f(i\leq s)$ for $i\leq s\leq j$.

Given $b,d\in \{0,\ldots,m\}$ such that $0\leq b  < d \leq  m$, let $I_{R}^{b,d}$ denote the persistence module $\{(I_{R}^{b,d})_i\}_{i=0}^m$ with $R$-coefficients such that 

\begin{align*}
    (I_{R}^{b,d})_i &= \begin{cases}
        R\,, & b\leq i<d \\
        0\,, & \text{otherwise}\,,
    \end{cases}\\
    I_{R}^{b,d}(i \le j): (I_{R}^{b,d})_i\to (I_{R}^{b,d})_j &= \begin{cases}
        \textrm{id}_R\,, & b\leq i\leq j<d \\
        0\,, & \text{otherwise}\,.
    \end{cases}
\end{align*}

We refer to the persistence modules $I_{R}^{b,d}$ as \emph{interval modules}.  An \emph{interval decomposition} of $f$ is a decomposition of $f$ into interval modules: 

\[
f \cong \bigoplus_k I_{R}^{b_k,d_k}\,,
\]

where $0\leq b_k<d_k\leq m$. Another way to view an interval decomposition is to choose a basis $\beta_i$ for each $f_i$ in a persistence module such 
that every structure map $f(i\leq j):f_i\to f_j$ maps elements of $\beta_i$ to $0$ or to basis elements of $\beta_j$ in a one-to-one manner. 
That is, given $v\in\beta_i$, precisely one of the following two statements holds:
\begin{enumerate}
    \item $f(i\leq j)(v) = 0$;
    \item $f(i\leq j)(v)\in \beta_j$ and $f(i\leq j)(\tilde{v}) = f(i\leq j)(v)$ implies that $\tilde{v}=v$ for any $\tilde{v}\in\beta_i$.
\end{enumerate} 

Equivalently, an interval decomposition arises if we can find a basis $\beta_i$ for each $f_i$ such that the matrix representation of $f(i \le j)$ with respect to $\beta_i$ and $\beta_j$ is a matching matrix\footnote{A matching matrix is a 0-1 matrix with at most one nonzero entry per row and column.}, for all $i \le j$; in this sense, an interval decomposition can be regarded as a choice of bases which simultaneously ``diagonalize'' all maps $f(i\le j)$.
The structure theorem of Gabriel \cite{gabriel_unzerlegbare_1972} is a general result describing the indecomposable isomorphism classes of quiver representations.  In the context of persistence modules, it implies the following result.

\begin{theorem}[Gabriel, \cite{gabriel_unzerlegbare_1972}]
\label{thm:gabriel}
Suppose that $R$ is a field and that $f$ is a persistence module over $R$.  Then
\begin{itemize}
    \item $f$ admits an interval decomposition; 
    \item two direct sums of nonzero interval modules are isomorphic --- that is, $\bigoplus_{k=1}^M I_{R}^{b_k,d_k} \cong \bigoplus_{\ell=1}^N I_{R}^{b'_\ell,d'_\ell}$ --- if and only if $M=N$ and there exists a permutation $\pi$ such that $(b_1, d_1), \ldots, (b_M,d_M)$ equals $(b'_{\pi(1)}, d'_{\pi(1)}), \ldots, (b'_{\pi(M)},d'_{\pi(M)})$.
\end{itemize}
\end{theorem}

By Theorem \ref{thm:gabriel}, when $R$ is a field, every persistence module with coefficients in $R$ admits an interval decomposition. 
However, interval decompositions for persistence modules with non-field coefficients are not guaranteed to exist (e.g., see the example in Section \ref{sec:motivation}).

\subsection{Persistent homology}

Interval decompositions of persistence modules provide a critical foundation 
to study
PH. 
In order to compute PH for a data set, one first constructs a \emph{filtration} $\mathcal{K}=\{\mathcal{K}_i\}_{i=0}^m$, which is a nested sequence of topological spaces that represent the data at different scales. 
As we increase the index $i$, which is called the \emph{filtration-parameter value}, new holes form and existing holes fill in. 
This information can be summarized in a \emph{persistence diagram}, which is a multiset $\textrm{PD}_n^F(\mathcal{K})=\{(b_k,d_k)\}_{k}$ of points in $\overline{\R}^2$ that detail the scale of the holes in the data set. Each point $(b_k,d_k)$ in a persistence diagram corresponds to a hole, where $b_k$ and $d_k$ are the filtration-parameter values that the hole is formed (birth filtration-parameter value) and is filled in (death filtration-parameter value), respectively.

In more detail, given a filtration $\mathcal{K}=\{\mathcal{K}\}_{i=0}^m$, one can take the $n$th homology with coefficients in a field $F$ to obtain a \emph{persistent homology module}, which is a persistence module $H_n(\mathcal{K};F)$ consisting of (1) a sequence $\{H_n(\mathcal{K}_i;F)\}_{i=0}^m$ of vector spaces (these are the \emph{homology groups} of $\mathcal{K}$ in dimension $n$) over $F$ that correspond to the topological spaces in $\mathcal{K}$ and (2) $F$-linear maps $\{\phi_i^j:H_n(\mathcal{K}_i;F) \to H_n(\mathcal{K}_j;F)\}_{0\leq i\leq j\leq m}$ (these are the structure maps) between homology groups that arise from the inclusion maps.
Given an interval decomposition 
\[
H_n(\mathcal{K};F) \cong \bigoplus_{k=1}^r I_{F}^{b_k,d_k}\,,
\]
the persistence diagram for $\mathcal{K}$ in dimension $n$ with $F$-coefficients is the multiset $\textrm{PD}_n^F(\mathcal{K})=\{(b_k,d_k)\}_{k=1}^r$.  The module $H_n(\mathcal{K};F)$ always admits an interval decomposition, as the coefficient ring is a field. 
This characterization formalizes the notion of a persistence diagram, which we discussed in Section \ref{sec:intro}.
A persistence module can also arise from a   \emph{persistent topological space}, which is a sequence $X_0 \to \cdots \to X_m$ of topological spaces and continuous maps.
Persistence modules of persistent topological spaces naturally arise when considering simplicial collapse for accelerating PH computations
\cite{boissonnat_strong_2018, boissonnat_edge_2020, glisse_swap_2022}. 

\subsection{Persistence modules beyond vector spaces}\label{sec:persistence_beyond}

The study of interval decompositions belongs to a large body of literature dedicated to understanding the isomorphism invariants of persistence modules \cite{carlsson_theory_2009, edelsbrunner_topological_2002, gabriel_unzerlegbare_1972}.
In \cite{patel_generalized_2018}, Patel generalized the notion of persistence diagrams to persistence modules of the form $f:\{0,\ldots,m\}\to C$, where $C$ is a symmetric monoidal category with images, or, under a slightly different construction,  where $C$ is an abelian category. 
These ideas were extended recently using Galois connections \cite{gulen_galois_2023} and  M\"obius homology \cite{patel_mobius_2023}. 
An important class of examples comes from 2-parameter persistence \cite{botnan_introduction_2023}; specifically, consider functors from the poset category of $\{0, \ldots, m\} \times \{0, \ldots, m\}$ into $C$. These objects are  common in applications and can be modeled as persistence modules of persistence modules, (i.e. functors $\{0, \ldots, m\} \to D$, where $D$ is the category of persistence modules with values in $C$).  

Researchers have also studied persistence beyond the setting of a single field.
In the Dionysus 2 package\footnote{See \url{https://www.mrzv.org/software/dionysus2/} for more information.}, Morozov implemented ``omni-field persistence,'' which computes persistent homology with coefficients across all fields $\mathbb{Z}_p$ where $p$ is a prime. 
(To our knowledge, no article has been written about omni-field persistence.)
Boissonnat and Maria \cite{boissonnat_computing_2019} introduced an algorithm that computes the persistent homology of a filtration with various coefficient fields in a single matrix reduction.
Obayashi and Yoshiwaki \cite{obayashi_field_2023} studied the dependence of persistence diagrams on the choice of field in the setting of filtrations (i.e., nested sequences of topological spaces). 
They introduced conditions on the homology of a filtration that are necessary and sufficient for the associated persistence diagrams to be independent of the choice of field\footnote{We will prove that these conditions are equivalent to the condition that the associated persistence module with $\Z$-coefficients splits as a direct sum of interval submodules (see Theorem \ref{thm:OurFieldIndependence}).} (see Section \ref{sec:motivation}). 
They produced an algorithm to verify field-choice independence by adapting a standard algorithm to compute persistent homology for simplex-wise filtrations\footnote{A \emph{simplex-wise filtration} on a simplicial complex $K = \{\sigma_1, \ldots,\sigma_m\}$ is a nested sequence of sub-simplicial complexes $\emptyset = \mathcal{K}_0 \subseteq \cdots \subseteq \mathcal{K}_m = K$ such that $\mathcal{K}_p = \{\sigma_1, \ldots, \sigma_p\}$ for all $p\in\{0,\ldots, m\}$.}.
They also conducted numerical experiments that demonstrate empirically that persistence diagrams rarely depend on the choice of field. 
Li et al. \cite{li_minimal_2021} reported similar empirical results.

There remain several gaps in the current literature. 
Although persistence diagrams have been defined for general families of functors $\{0,\ldots, m\} \to C$ \cite{patel_generalized_2018}, it is only known that interval decompositions exist when $C$ is a category of finite-dimensional vector spaces (see Theorem \ref{thm:gabriel}).  
While the present work does not explicitly study PH or persistence diagrams, our results concerning the fundamental structure of persistence modules have ramifications in these areas. 
In particular, we complement the results of \cite{obayashi_field_2023} by linking field-independence to the algebraic structure (concretely, the indecomposable factors) of a persistence module, rather than the topology of an underlying simplicial complex.
Such a perspective is important, as it helps one to relate topological ideas involving filtrations with algebraic ideas involving the corresponding persistence modules. 
By better understanding persistence modules, we also gain a better grasp of their relation to PH, as well as generalizations thereof.

\section{Motivating Applications}
\label{sec:applications}

\subsection{Pruning a persistence module}\label{sec:pruning}

Let $f:\{0,\ldots,m\}\to R\textrm{-Mod}$ be a persistence module that is pointwise free and finitely-generated. 
For a subset $P\subseteq \{0,\ldots,m\}$, we define the \emph{pruning} of $f$ at $P$ as the restricted persistence module $f|_{\{0,\ldots,m\}\setminus P}: \{0,\ldots,m\}\setminus P\to R\textrm{-Mod}$, which we denote $f\setminus P$.

Suppose $f:\{0,\ldots,m\}\to R\textrm{-Mod}$ is a persistence module that is pointwise free and finitely-generated that does not admit an interval decomposition. 
A natural question that arises is how can one prune $f$ to admit an interval decomposition. 
That is, how can one find a subset $P\subseteq \{0,\ldots,m\}$ such that $f\setminus P$ admits an interval decomposition?
More generally, how can one characterize all subsets $P\subseteq \{0,\ldots,m\}$ such that $f\setminus P$ admits an interval decomposition?

By Theorem \ref{thm:problem_statement}, we know that a persistence module $f$ admits an interval decomposition if and only if the cokernel of every structure map $f(i\leq j)$ is free. 
This motivates the following.
\begin{proposition}\label{prop:pruning}
    Let $f:\{0,\ldots,m\}\to R\textrm{-Mod}$ be a persistence module that is pointwise free and finitely-generated, and let $P\subseteq \{0,\ldots,m\}$.
    Then the pruned persistence module $f\setminus P$ admits an interval decomposition if and only if the structure map $f(i\leq j)$ has free cokernel for all $i,j\notin P$. 
\end{proposition}

\begin{proof}
    Suppose the structure map $f(i\leq j)$ has free cokernel for all $i,j\notin P$. 
    Then by Theorem \ref{thm:problem_statement}, we have that $f\setminus P$ admits an interval decomposition.
    Similarly, if $f\setminus P$ admits an interval decomposition, then  Theorem \ref{thm:problem_statement} tells us that $f(i\leq j)$ has a free cokernel for all $i,j\in \{0,\ldots,m\}\setminus P$. 
\qed \end{proof}

Proposition \ref{prop:pruning} gives a characterization of the sets $P$ such that $f\setminus P$ admits an interval decomposition. 
Namely, whenever the structure map $f(i\leq j)$ does not have a free cokernel, either $i\in P$ or $j\in P$. 

\subsection{Persistent homology and independence of field choice}
\label{sec:motivation}

Theorem \ref{thm:problem_statement} allows us to deduce that the persistence diagram of a filtered topological space is independent of the choice of coefficient field if and only if the associated  persistence module over $\Z$ splits as a direct sum of interval submodules (subject to the condition that homology in the next lowest dimension is torsion-free; see Theorem \ref{thm:OurFieldIndependence}). One direction of this result is unsurprising: if there is no torsion in the next lowest dimension, then  universal coefficients are ``natural,'' so one might reasonably expect field-independence to follow from a splitting of the module over $\Z$ into  interval modules. The converse is more surprising.  Our proof relies on a result of Obayashi and Yoshiwaki \cite[Theorem 1.9]{obayashi_field_2023}, which shows equivalence of conditions 3 and 4 in Theorem \ref{thm:OurFieldIndependence}.

\begin{theorem}
\label{thm:OurFieldIndependence}
    Let $\mathcal{K}=\{\mathcal{K}_i\}_{i=0}^m$ be a filtration of topological spaces.
    Suppose that the homology groups $H_{k-1}(\mathcal{K}_a;\Z)$ and $H_{k}(\mathcal{K}_a;\Z)$ with integer coefficients are free and finitely-generated for every $a$ such that $0\leq a\leq m$. Then the following are equivalent.
    \begin{enumerate}
        \item The persistence module $H_k(\mathcal{K};\Z)$ splits as a direct sum of interval submodules.    
        \item For all $a$ and $b$ such that $0 \le a \le b \le m$, the cokernel of the map $\phi_a^b:H_k(\mathcal{K}_a;\Z) \to H_k(\mathcal{K}_b;\Z)$ induced by the inclusion is free.
        \item (Obayashi and Yoshiwaki) For all $a$ and $b$ such that $0 \le a \le b \le m$, the relative homology group $H_k(\mathcal{K}_b,\mathcal{K}_a;\Z)$ is free. 
        \item (Obayashi and Yoshiwaki) The persistence diagram $\textrm{PD}_k^F(\mathcal{K})$ of dimension $k$ is independent of the choice of coefficient field $F$.   
    \end{enumerate}
\end{theorem}

\begin{proof}
    Obayashi and Yoshiwaki\cite[Theorem 1.9]{obayashi_field_2023} showed the equivalences between 3 and 4. Theorem \ref{thm:problem_statement} gives the equivalence between 1 and 2. It remains to show the equivalence between 2 and 3. 

    We have the following long exact sequence for relative homology:
    $$\cdots\to H_k(\K_a;\Z)\stackrel{\phi_a^b}\to H_k(\K_b;\Z)\stackrel{j_*}\to H_k(\K_b,\K_a;\Z)\stackrel{\partial}\to H_{k-1}(\K_a;\Z)\to\cdots\,,$$
    from which we can extract the following short exact sequence 
    $$0\to \mathrm{coker}(\phi_a^b)\to H_k(\K_b,\K_a;\Z)\to \im(\partial)\to 0\,.$$

    Suppose that $\mathrm{coker}(\phi_a^b)$ is free.
    Because $\im(\partial)$ is free (it is a subgroup of the free abelian group $H_{k-1}(\K_a;\Z)$), this short exact sequence splits, which implies that $H_k(\K_b,\K_a;\Z)\cong \mathrm{coker}(\phi_a^b)\oplus\im(\partial)$. Because $\mathrm{coker}(\phi_a^b)$ is free, $H_k(\K_b,\K_a;\Z)$ must also be free. 

    Conversely, if $H_k(\K_b,\K_a;\Z)$ is free, then $\mathrm{coker}(\phi_a^b)$ must also be free because it injects into a free abelian group.
\qed \end{proof}

\begin{rmk} 
In a prior version of \cite{obayashi_field_2023} (which is available on arXiv \cite{obayashi2020field}), Obayashi and Yoshiwaki gave a constructive proof of the equivalence of conditions 1 and 3 in the special case in which $\{\mathcal{K}_i\}_{i=0}^m$ is a simplex-wise filtration on a simplicial complex \cite[Lemma 2]{obayashi2020field}.
The proof appears to generalize to any cell-wise filtration on a regular CW complex. 
However it relies on the assumption that $\mathcal{K}_i$ contains exactly $i$ distinct simplices (or more generally, cells), for all $i$.
\end{rmk}

\begin{rmk}
    The omni-field persistence algorithm\footnote{See Section \ref{sec:persistence_beyond}.}, which is a part of the persistent-homology package Dionysus 2, is especially relevant to Theorem \ref{thm:OurFieldIndependence}. In particular, it computes persistent homology over all fields $\Z_p$, where $p$ is a prime, and indicates when the persistent diagram of a filtration depends on one's choice field.
    To the best of our knowledge, this algorithm can be applied only to simplex-wise filtrations.    
\end{rmk}

We illustrate the relation between the field-choice independence of a persistence diagram of a filtration and the existence of an interval decomposition of the corresponding persistence module with $\Z$ coefficients. 
Consider the following persistent topological space $\mathcal{K}$: 
$$\bullet \hookrightarrow S^1 \stackrel{f}\to S^1 \hookrightarrow \C \,,$$
where $\bullet$ is a single point, $S^1 = \{z\in\C|\|x\|=1\}$, and $f(z) = z^2$.
By taking the first homology over a field $F$, we obtain the following persistence module: 
$$0\to F\stackrel{\cdot2}\to F\to 0\,.$$
If $\textrm{char}(F)=2$, the multiplication-by-$2$ map is the zero map, so we can write 
\[
(0\to F\stackrel{\cdot2}\to F\to 0) \cong (0\to F \to 0 \to 0) \oplus (0\to 0\to F\to 0)\,.
\]
Otherwise, the multiplication-by-$2$ map is an isomorphism and the persistence module in question is an interval module. 
Therefore, when indexing from $0$, the corresponding persistence diagram $\textrm{PD}_1^F(\mathcal{K})$  for first homology depends on the characteristic of the coefficient field:
\[
\textrm{PD}_1^F(\mathcal{K}) = \begin{cases}
    \{(1,2),(2,3)\}\,, & \textrm{char}(F)=2\,, \\
    \{(1,3)\}\,, & \text{otherwise}\,.
\end{cases}
\]

When taking the first homology of $\mathcal{K}$ with integer coefficients, the corresponding persistence module is 
\[
0\to \Z\stackrel{\cdot2}\to\Z\to 0\,,
\]
which does not decompose into interval modules.

\begin{rmk}
    Proposition \ref{prop:pruning}, which we formulated in Section \ref{sec:pruning}, discusses how one can prune a persistence module such the resultant persistence module admits an interval decomposition. 
    By Theorem \ref{thm:OurFieldIndependence}, identifying parameter values to prune, such that the resultant persistence module has an interval decomposition, also allows us to pinpoint where field-choice independence fails. 
    In particular, the parameter values to prune a persistence module in order to yield an interval decomposition are precisely the parameter values that field-choice independence fails.
    That is, removing these indices would result in a filtration whose persistence diagram is independent of field choice.    
\end{rmk}

\subsection{Universal Cycle Representatives}\label{sec:universal_cycle_reps}

In Section \ref{sec:intro}, we introduced the idea of a \emph{universal cycle basis} for persistent homology. Here we make that idea precise and explain how such a basis can be obtained with the methods introduced in this paper.

Suppose that $\mathcal{K}=\{\mathcal{K}_i\}_{i=0}^m$ is a nested sequence of finite cell complexes (simplicial, cubical, or otherwise). Let $Z$ be a collection of cycles in $\mathcal{K}_m$ with integer coefficients, and for each $0 \le i \le m$, let $Z_i$ denote the set of cycles $z \in Z$ such that (i) $z$ is a cycle in $\mathcal{K}_i$, and (ii) $z$ is not a boundary in $\mathcal{K}_i$. 
We further require that each $z\in Z$ must belong in some $Z_i$.
Thus $\{ [z] : z \in Z_i\}$ is a well-defined set of integer homology classes in $\mathcal{K}_i$. We call $Z_i$ an integer cycle basis for $\mathcal{K}_i$ if the homology group $H_k(\mathcal{K}_i; \Z)$ is free and $\{ [z] : z \in Z_i\}$ is an integer cycle basis for $H_k(\mathcal{K}_i; \Z)$. We call the larger set $Z$ a \emph{integer cycle basis} for the persistent homology module $H_{k}(\mathcal{K}; \Z)$ as a whole if, for each $i$, the set  $Z_i$ is an integer cycle basis of $\mathcal{K}_i$.

The notion of an integer cycle basis extends naturally to the notion of an \emph{$F$-cycle basis}, for any coefficient ring $F$; we simply replace coefficients in homology with $F$. It's natural to ask how $\Z$-cycle bases and $F$-cycle bases might relate to one another, for the same underlying filtered simplicial complex $\mathcal{K}$. The answer to this question can be complex in general, but there is a special case where the answer is straightforward. Suppose that the map $\phi_i: H_{k}(\mathcal{K}_i;\Z) \otimes F \to H_{k}(\mathcal{K}_i;F)$ is an isomorphism for all $i$ (e.g., if the integer homology for each $\K_i$ is torsion-free). These isomorphisms are then natural, in the sense that the following diagram commutes.

\[\begin{tikzcd}
	{H_{k}(\mathcal{K}_0;\Z)} & \cdots & {H_{k}(\mathcal{K}_m;\Z)} \\
	{H_{k}(\mathcal{K}_0;\Z) \otimes F} & \cdots & {H_{k}(\mathcal{K}_m;\Z) \otimes F} \\
	{H_{k}(\mathcal{K}_0;F) } & \cdots & {H_{k}(\mathcal{K}_m;F) }
	\arrow[from=1-1, to=1-2]  
	\arrow["\otimes F",from=1-1, to=2-1]      
	\arrow[from=1-2, to=1-3]
	\arrow["\otimes F",from=1-3, to=2-3]  
	\arrow[from=2-1, to=2-2]
	\arrow["\cong"', from=2-1, to=3-1]
	\arrow["\phi_0", from=2-1, to=3-1]    
	\arrow[from=2-2, to=2-3]
	\arrow["\cong"', from=2-3, to=3-3]
	\arrow["\phi_m", from=2-3, to=3-3]      
	\arrow[from=3-1, to=3-2]
	\arrow[from=3-2, to=3-3]
\end{tikzcd}\]

The vertical maps in this diagram are basis-preserving, so if $Z$ is an integer cycle basis for $\mathcal{K}$, then $Z \otimes F = \{ z \otimes F : z \in Z\}$ is an $F$-cycle basis for $\mathcal{K}$.  
Note that if $z = \sum_i z_i \sigma_i$ is a linear combination of simplices, then the vector $z \otimes F$ is simply $\sum_i (z_i \otimes F) \sigma_i$.

Under these conditions, we call the cycle basis $Z$ \emph{universal}, in the sense that $\{ z \otimes F : z \in Z\}$ is a cycle basis for the persistent homology module $H_{k}(\mathcal{K};F)$ for every coefficient field $F$. 
Obayashi and Yoshiwaki provided in \cite[Lemma 2]{obayashi2020field} (which is a preprint version of \cite{obayashi_field_2023}) an algorithm to compute such a cycle basis if one exists, under the condition that $\mathcal{K}$ be a simplex-wise filtration.\footnote{The approach in \cite[Lemma 2]{obayashi2020field}  is sketched in the proof of Theorem \ref{thm:yoshiwakiAlgorithmCorrect}.}

The present paper provides a simple alternative which works for any filtration, unconstrained by the number of simplices added at each step. 
The process begins by obtaining a matrix representation of the persistence module $H_{k}(\mathcal{K}_0;\Z) \to \cdots \to H_{k}(\mathcal{K}_m;\Z)$. 
This can be accomplished by applying standard methods --- primarily Smith Normal matrix factorization --- to obtain, first, a basis for each homology group, and second, a matrix representation $M_i$ for each homomorphism $H_{k}(\mathcal{K}_i;\Z) \to H_{k}(\mathcal{K}_{i+1};\Z)$. 

Our next step is to check for the presence of torsion, which precludes the existence of a universal cycle basis in following two cases: (1) when $H_k(\mathcal{K}_{i}; \Z)$ has torsion for some $i$ and (2) when $H_{k-1}(\mathcal{K}_{i}; \Z)$ has torsion for some $i$. 
In the first case, when $H_k(\mathcal{K}_{i}; \Z)$ has torsion for some $i$, then the homology group $H_k(\mathcal{K}_{i}; \Z)$ is not free, and has no basis. 
Therefore no universal cycle basis exists for the persistence module in question. 
In the second case, suppose $H_{k-1}(\mathcal{K}_{i}; \Z)$ has torsion modulo $p$ for some index $i$ and some prime $p$. 
In this case, the map $H_{k}(\mathcal{K}_i;\Z) \xrightarrow{\otimes \mathbb{F}_p} H_{k}(\mathcal{K}_i;\mathbb{F}_p)$ is not an isomorphism, where $\mathbb{F}_p$ denotes the finite field of order $p$. In particular, if $H_{k-1}(\mathcal{K}_{i}; \Z)$ has torsion modulo $p$, then by the Universal Coefficient Theorem \cite{hatcher}, the dimension of $H_{k}(\mathcal{K}_i;\Z)$ as a free module over $\Z$ will be strictly less than the dimension of $H_{k}(\mathcal{K}_i;F_p)$ as a vector space over $F_p$. Thus no universal cycle basis can exist. 

Our final step is to apply the decomposition algorithm developed in this paper. If torsion is not present in dimensions $k$ or $k-1$, then by the Universal Coefficient Theorem \cite{hatcher}, the map $H_k(\mathcal{K}_{i}; \Z) \otimes F\to H_k(\mathcal{K}_{i}; F)$ is an isomorphism for every choice of coefficient field $F$. 
In this case, we can apply 
Algorithm \ref{alg:matrix} to the persistence module represented by the sequence of matrices $\bullet \xrightarrow{M_0} \bullet \xrightarrow{M_1} \cdots \xrightarrow{M_{m-1}} \bullet$ to obtain a sequence of bases corresponding to an interval decomposition (if one exists). 
If this module has an interval decomposition (i.e., if the cokernel of every structure map is free), then the algorithm will return a corresponding collection of bases; each basis vector represents a linear combination of homology classes $\sum_z \alpha_z [z]$, which is represented by a cycle $\sum_z \alpha_z z$.  If, for each interval submodule $ \cdots \to 0 \to V_i \xrightarrow{\cong} \cdots \xrightarrow{\cong} V_j \to 0 \to \cdots $ in the interval decomposition, we denote the cycle representative corresponding to $V_i$ by $z(V)$, then the resulting collection of cycles $\{ z(V): V~\text{is a summand of the interval decomposition}\}$ is an integer cycle basis for $H_{k}(\mathcal{K};\Z)$.

\begin{rmk}
    We have shown that a universal cycle basis exists for a filtered cell complex whenever 1) the integer homology in dimensions $k-1$ and $k$ are free and 2) the structure maps of the persistence module (with integer coefficients) have free cokernel. 
    These conditions also guarantee that the corresponding persistence diagram is independent of field choice (see Theorem \ref{thm:OurFieldIndependence}).
\end{rmk}

\section{Uniqueness of Interval Decompositions}
\label{sec:uniqueness_of_interval_decompositions}

In this section, we prove that interval decompositions of persistence modules that are pointwise free and finitely-generated over a PID are unique when they exist. 

\begin{framed}
We fix a PID $R$ and a persistence module $f$ that is pointwise free and finitely-generated over $R$, which we will use for the rest of this paper.
We additionally assume, without loss of generality, that $f_0=f_m=0$.
\end{framed}

\begin{lemma}\label{lemma:interval_decomp_tensor}
    Let $f:\{0,\ldots,m\}\to R\textrm{-Mod}$ be a persistence module that is pointwise free and finitely-generated over $R$. Suppose that $f$ admits an interval decomposition 
    \[
    f\cong \bigoplus_k I_{R}^{b_k,d_k}\,.
    \]
    Let $F=\mathrm{Frac}(R)$ be the field of fractions of $R$, and consider the persistence module $f\otimes F$, which is the persistence module over $F$ that is given by 
    \begin{align*}
        (f\otimes F)_a &=f_a\otimes F\,, \\
        (f\otimes F)(a\leq b) &= f(a\leq b)\otimes \mathrm{id}_F\,.
    \end{align*}
    Then $\bigoplus_k I_{F}^{b_k,d_k}$ is an interval decomposition of $f\otimes F$. 
\end{lemma}

\begin{proof}
    Observe that $I_R^{b,d}\otimes F = I_F^{b,d}$. 
    Because tensor product distributes over direct sum, we have
    \begin{align*}
        f\otimes F &\cong \left(\bigoplus_k I_R^{b_k,d_k}\right)\otimes F \\
        &= \bigoplus_k (I_R^{b_k,d_k}\otimes F)\\
        &= \bigoplus_k I_F^{b_k,d_k}\,,
    \end{align*}
    as desired.
\qed \end{proof}

\begin{theorem}\label{thm:unique_interval_decomp}
    Let $f:\{0,\ldots,m\}\to R\textrm{-Mod}$ be a persistence module that is pointwise free and finitely-generated over $R$, and suppose $f$ admits an interval decomposition. Then the associated multiset of intervals is unique.
\end{theorem} 

\begin{proof}
    Let $F=\mathrm{Frac}(R)$ be the field of fractions of $R$.
    Suppose that $\bigoplus_{k=1}^{r} I_R^{b_k,d_k}$ and $\bigoplus_{\ell=1}^{r'}   I_R^{b_\ell',d_\ell'}$ are both interval decompositions of $f$. By Lemma \ref{lemma:interval_decomp_tensor}, both $\bigoplus_{k=1}^{r} I_F^{b_k,d_k}$ and $\bigoplus_{\ell=1}^{r'} I_F^{b_\ell',d_\ell'}$ are interval decompositions 
    of $f\otimes F$. Theorem \ref{thm:gabriel} guarantees that interval decompositions of persistence modules over $F$ are unique up to permutation, so $r=r'$ and there exists a permutation $\pi$ such that  $(b_k,d_k)=(b_{\pi(k)}',d_{\pi(k)}')$ for every $k$.
\qed \end{proof}

\begin{rmk}
    Theorem \ref{thm:unique_interval_decomp} extends to persistence modules with coefficients in any integral domain.
\end{rmk}

\section{Proof of Theorem \ref{thm:problem_statement}: Necessity}
\label{sec:interval_decomp}

We now prove the main theoretical result, which we repeat here for ease of reference. 

{
\renewcommand{\thetheorem}{\ref{thm:problem_statement}}
\begin{theorem}
Let $f$ be a persistence module that is pointwise free and finitely-generated over $R$. 
Then $f$ splits into a direct sum of interval modules if and only if the cokernel of every structure map $f(a\leq b)$ is free.
\end{theorem}
\addtocounter{theorem}{-1}
} 

\begin{framed}
    The condition that $f(a \le b)$ has a free cokernel is equivalent to several other well-known conditions, such as:
    \begin{itemize}
        \item The module $f_b$ splits as a direct sum $I \oplus C$ for some submodule $C\subseteq f_b$, where $I$ is the image of $f(a \le b)$.
        \item In the language of Smith normal form (see Appendix \ref{appendix:smith_normal_form}), the map $f(a \le b)$ has unit elementary divisors.  That is, if $A$ is the matrix representation of $f(a\le b)$ with respect to some pair of bases for $f_a$ and $f_b$, and if $SAT = D$ is the Smith normal form of $A$, then the nonzero diagonal entries of $D$ are units.
    \end{itemize}
\end{framed}

The proof of Theorem \ref{thm:problem_statement} has two halves: necessity and sufficiency.  Necessity is straightforward and is established in Lemma \ref{Lemma:easy_direction}.  Sufficiency requires greater effort and is established in Section \ref{sec:interval_decomp_hard}.

\begin{lemma}[Necessity]
\label{Lemma:easy_direction}
If $f$ splits into a direct sum of interval modules, then the cokernel of every structure map $f(a\leq b)$ is free. 
\end{lemma}

\begin{proof}
Suppose that we can write $f=h^1\oplus h^2\oplus \cdots\oplus h^r$, where each $h^k$ is an interval module. 
Fix $a$ and $b$ such that $0\leq a\leq b\leq m$, and let $N=\{k:h^k_b\neq0\}$ be the indices of the interval modules that are nonzero at $b$. 
Let $X=\{k\in N:h^k_a\neq 0\}$ (i.e., the set of indices of intervals including $[a,b]$) and $Y=\{k\in N:h^k_a= 0\}$ (i.e., the set of indices of intervals that are $0$ at $a$). In particular, $\bigoplus_{k\in X} h^k_b=\im(f(a\leq b))$, because $X$ is precisely the set of indices for the intervals that contain $[a,b]$.

We note the following: 
\begin{align*}
    f_b &= \bigoplus_{k=1}^r h^k_b \\
    &= \bigoplus_{k\in N} h^k_b \\
    &=\bigoplus_{k\in X} h^k_b \oplus \bigoplus_{k\in Y} h^k_b\\
    &=\im(f(a\leq b)) \oplus \bigoplus_{k\in Y} h^k_b\,.
\end{align*}

Because $f_b$ is free and submodules of free modules are free, it follows that $\Coker(f(a\leq b))=f_b/\im(f(a\leq b))\cong \bigoplus_{k\in Y} h^k_b$ is free. 
\qed \end{proof}

\section{Proof of Theorem \ref{thm:problem_statement}: Sufficiency}
\label{sec:interval_decomp_hard}

We now build the framework to prove sufficiency of Theorem \ref{thm:problem_statement}. That is, we prove the following theorem.

\begin{theorem}[Sufficiency]
\label{thm:hard_direction_result}
    Let $f$ be a persistence module that is pointwise free and finitely-generated over $R$, and suppose that the cokernel of every structure map $f(a\leq b)$ is free. Then $f$ splits into a direct sum of interval modules.
\end{theorem}
 
To prove Theorem \ref{thm:hard_direction_result}, we introduce Algorithm \ref{alg:simple}, which produces an interval decomposition for a persistence module $f$ whenever the cokernel of every structure map is free. 
That is, we inductively construct a basis $\beta_i$ for each $f_i$ that collectively yield an interval decomposition.  
Throughout this section, we fix a (finitely-indexed) pointwise free and finitely-generated persistence module $f:\{0,\ldots,m\}\to R$ such that every structure map $f(a\leq b)$ has free cokernel. 
Recall our assumption that, without loss of generality, $f_0=f_m=0$ (see Section \ref{sec:uniqueness_of_interval_decompositions}). 

For a non-negative integer $n$, let $\mathbf{n}=\{1,\ldots,n\}$. 
This definition holds even when $n = 0$; in this case, $\textbf{0} = \emptyset$.  
We note that $f$ is indexed by $\{0,1,\ldots,m\}$, which can also be written as $\mathbf{m} \cup \{0\}$.
Because $f_0 = f_m = 0$, every map of form $f(0 \le p)$ or $f(p \le m)$ is a zero map.

\begin{definition}\label{def:Im_and_Ker}
    Fix $a\in\mathbf{m}$, and let $x,y\in\mathbf{m}\cup\{0\}$. Define $$\mathrm{Ker}[a,y]=\begin{cases}
        \ker(f(a\leq y))\,, & a\leq y \\ 
        0\,, & \text{otherwise}
    \end{cases}$$ and $$\mathrm{Im}[x,a]=\begin{cases}
        \im(f(x\leq a))\,, & x\leq a \\
        f_a\,, & \text{otherwise}\,.
    \end{cases}$$
\end{definition}

Note that $\Ker[a,m]=f_a$ because $f_m = 0$ and $\Ker[a,0]=0$ because $0 \le a$.  Thus, for each $a$, we have a nested sequence of submodules $0 = \Ker[a,0] \su \cdots \su \Ker[a,m]= f_a$ beginning with $0$ and ending with $f_a$ that are the kernels of the maps $f(a\leq x)$. Similarly, the submodules $\im[x,a]$ form a nested sequence $0 = \im[0,a] \su \cdots \su \im[m,a]= f_a$ of submodules beginning with $0$ and ending with $f_a$ that are the images of the maps $f(x\leq a)$.

\subsection{Complements}
\label{sec:complements}

We now discuss the existence of complements for modules that are relevant for Algorithm \ref{alg:simple}.

\begin{definition}
    We say that $C$ \emph{complements} $A \su B$ (or that $C$ is a \emph{complement} of $A$ in $B$) whenever $A \oplus C = B$.
\end{definition}

For convenience, we adopt the notation 
$$
\mathrm{Coker}[x,a]=f_a/\mathrm{Im}[x,a]\,.
$$

For the remainder of this section, we assume that the cokernel of every structure map $f(a \le b)$ is free.  Thus,
$$
f_b\cong \mathrm{Im}[a,b]\oplus \Coker[a,b]\,.
$$

\begin{rmk}
    It is readily checked that every $\Ker[a,y]$ and $\im[x,a]$ has a complement in $f_a$. 
\end{rmk}

\begin{rmk}
    In our proof of Lemma \ref{lemma:complement_equivalence}, we will use the fact that a finitely-generated module over a PID is free if and only if it is torsion-free.
\end{rmk}

\begin{lemma}\label{lemma:complement_equivalence}
Let $G$ be a free and finitely-generated module over $R$, and let $A\subseteq G$. Then the following are equivalent: 
\begin{enumerate}
    \item[(1)] $A$ has a complement in $G$;
    \item[(2)] $G/A$ is torsion-free;
    \item[(3)] if $mz\in A$ for $z\in G$ and nonzero $m\in R$, then $z\in A$.
\end{enumerate}
\end{lemma}

\begin{proof}
This equivalence (1)--(3) follows from standard properties of finitely-generated modules over a PID.

To see that (1) implies (2), we note that if $G=A\oplus B$, then $G/A\cong B$, which is free (and therefore torsion-free), as it is a submodule of a free module over a PID. 

We now show that (2) implies (1). 
Let $\phi:G\to G/A$ denote the projection map. Let $\{\tilde{g}_i\}_{i\in I}$ be a basis of $G/A$ (we can choose such a basis because $G/A$ is torsion-free and therefore free), and choose $g_i\in \phi^{-1}(\tilde{g_i})$. 
Let $B\subseteq G$ be the submodule generated by $\{g_i\}_{i\in I}$. 
We claim that $A\oplus B = G$. Any nontrivial linear combination of the elements $g_i$ yields a nontrivial linear combination of elements $\tilde{g}_i$ under $\phi$; this linear combination is nonzero because the $\tilde{g}_i$ are linearly independent.  
Therefore, no nontrivial linear combination of elements $g_i$ lies in $\ker(\phi)=A$, so $A\cap B=0$. 
We now show that $A+B=G$. 
Take $g\in G$, and write $\phi(g)=\sum_{i\in I} c_i\tilde{g}_i=\sum_{i\in I} c_i\phi(g_i)$ for some $c_i\in R$. 
From this, we see that $g=\sum_{i\in I} c_ig_i+a$ for $a\in\ker(\phi)=A$, as desired. 

We now show that (2) implies (3). 
Suppose that $mz\in A$ with nonzero $m\in R$ and $z\in G$. 
If $z\in G\setminus A$, then $z+A\in G/A$ is nonzero, but $mz+A=0$.
This implies that $z+A$ is in the torsion of $G/A$, which is a contradiction. 

We now show that (3) implies (2). 
If $G/A$ has torsion, then there is a nonzero $z+A$ such that $mz+A=0$ for a nonzero $m\in R$. 
This implies that $z\in G\setminus A$ and that $mz\in A$.
This shows the desired result by contrapositive, thus completing the proof.
\qed \end{proof}

\begin{theorem}\label{thm:complement_special}
Fix $a,b\in\mathbf{m}$ such that $a\leq b$. For any $x,y\in\mathbf{m}\cup\{0\}$, the submodule $\Ker[a,y]+\mathrm{Im}[x,a] \subseteq f_a$ has a complement in $f_a$. 
\end{theorem}
\begin{proof}
We first note that if $y\leq a$, then $\Ker[a,y]+\mathrm{Im}[x,a]=\mathrm{Im}[x,a]$, which we know has a complement in $f_a$ that is isomorphic to $\Coker[x,a]$. 

If $a\leq x$, then $\Ker[a,y]+\mathrm{Im}[x,a] = f_a$ has a trivial complement $C = 0$ in $f_a$.

We now need to prove the result for $x\leq a \leq y$. Because $f_a$ is free, it is sufficient to show that the quotient $\frac{f_a}{\Ker[a,y]+\mathrm{Im}[x,a]}$ is free. 

Let $C$ be a complement of $\mathrm{Im}[a,y]$ in $f_y$.  Then we have a chain of isomorphisms.

\begin{align}
    \Coker[x,y]&\cong \frac{f_y}{\mathrm{Im}[x,y]} 
    \notag
    \\
    &= \frac{\mathrm{Im}[a,y]\oplus C } 
    {\mathrm{Im}[x,y]}
    \notag
    \\
    &\cong \frac{\mathrm{Im}[a,y]}{\mathrm{Im}[x,y]}\oplus C 
    \label{eq_greg_tag_a}
    \\
    &\cong \frac{f_a}{\Ker[a,y]+\mathrm{Im}[x,a]} \oplus C\,. 
    \label{eq_greg_tag_b}
\end{align}

The last two isomorphisms (i.e., \eqref{eq_greg_tag_a} and \eqref{eq_greg_tag_b}) must be proved.

The isomorphism \eqref{eq_greg_tag_a} holds because $\mathrm{Im}[x,y] \subseteq \mathrm{Im}[a,y]$.  

For the isomorphism \eqref{eq_greg_tag_b}, recall that $x\leq a \leq y$ and that the following hold by definition:
\begin{itemize}
    \item $\Ker[a,y] = \ker(f(a\leq y))$; 
    \item $\mathrm{Im}[x,a]=\mathrm{Im}(f(x\leq a))=f(x\leq a)(f_a)$;
    \item $\mathrm{Im}[x,y]=\mathrm{Im}(f(x\leq y))=f(x\leq y)(f_x) = f(a\leq y)( f(x\leq a)(f_x))$;
    \item $\mathrm{Im}[a,y]=\mathrm{Im}(f(a\leq y))=f(a\leq y)(f_a)$. 
\end{itemize}

We now show the following chain of isomorphisms:
\begin{align}
\frac{\mathrm{Im}[a,y]}{\mathrm{Im}[x,y]} &\cong \frac{f(a\leq y)(f_a)}{f(a\leq y)( f(x\leq a)(f_x))} \label{eq:complement_special_1} \\
&\cong \frac{f(a\leq y)(f_a)}{f(a\leq y) \Big ( f(x\leq a)(f_x)+\ker(f(a\leq y)) \Big)} \label{eq:complement_special_2}\\
&\cong \frac{f_a/\ker(f(a\leq y))}{\Big(f(x\leq a)(f_x)+\ker(f(a\leq y))\Big)/\ker(f(a\leq y))} \label{eq:complement_special_3}\\
&\cong \frac{f_a}{f(x\leq a)(f_x)+\ker(f(a\leq y))} \label{eq:complement_special_4}\\
&\cong \frac{f_a}{\Ker[a,y]+\mathrm{Im}[x,a]}\,.\label{eq:complement_special_5}
\end{align}
The isomorphism \eqref{eq:complement_special_2} follows from the inequality
$$f(a\leq y)( f(x\leq a)(f_x))=f(a\leq y) \Big( f(x\leq a)(f_x)+\ker(f(a\leq y)) \Big )\,.$$
The isomorphism \eqref{eq:complement_special_3} is a result of the first isomorphism theorem (See \cite[Section 10.2, Theorem 4.1]{dummit_abstract_2004}).
The isomorphism \eqref{eq:complement_special_4} is a result of the third isomorphism theorem (See \cite[Section 10.3, Theorem 4.1]{dummit_abstract_2004}). 
This establishes  \eqref{eq_greg_tag_b}.  
The isomorphism \eqref{eq_greg_tag_b} shows that $\frac{f_a}{\Ker[a,y]+\mathrm{Im}[x,a]}$ can be viewed as a submodule of $\Coker[x,y]$, which is free by hypothesis.  
Because submodules of free modules (over PIDs) are free, $\frac{f_a}{\Ker[a,y]+\mathrm{Im}[x,a]}$ is free.
Therefore, by Lemma \ref{lemma:complement_equivalence}, the inclusion ${\Ker[a,y]+\mathrm{Im}[x,a]} \su f_a$ has a free cokernel, so the submodule $\Ker[a,y]+\mathrm{Im}[x,a]$ has a complement in $f_a$, as desired. 
\qed \end{proof}

\begin{lemma}\label{lemma:intersection_closure}
If the submodules $H, K \su f_a$ each admit complements, then so does $H \cap K$.
\end{lemma}

\begin{proof}
Suppose that $mz\in H\cap K$ for nonzero $m\in R$ and $z\in f_a$. 
Because $mz\in H$ and $H$ admits a complement, we see by Lemma \ref{lemma:complement_equivalence} that $z\in H$. Similarly, $z\in K$, which implies that $z\in H\cap K$. 
Therefore, by Lemma \ref{lemma:complement_equivalence}, $H\cap K$ has a complement in $f_a$.
\qed \end{proof} 

\begin{lemma}\label{lemma:complement_shrink}
If $L\subseteq H\subseteq f_a$ and $L$ admits a complement in $f_a$, then $L$ admits a complement in $H$.
\end{lemma}

\begin{proof}
By Lemma \ref{lemma:complement_equivalence}, it is enough to show that $H/L$ is free. 
Because $H/L\subseteq f_a/L$ and $f_a/L$ is free, we must have that $H/L$ is free as well.  
\qed \end{proof}

\subsection{Algorithm for Computing an Interval Decomposition}
\label{sec:algorithm}

We now provide a simple, high-level presentation of an algorithm to compute interval decompositions for persistence modules that are pointwise free and finitely-generated over $R$, which completes the proof of Theorem \ref{thm:hard_direction_result}.
We provide a detailed description in the language of matrix algebra in Section \ref{sec:matrix_algorithm}. 

Let $f:\{0\}\cup\mathbf{m}\to R\textrm{-Mod}$ be a persistence module that is pointwise free and finitely-generated over $R$. 
Recall that $f_0=f_m=0$ by convention. Our task will be to define a sequence of bases $\beta_i \su f_i$ such that $f(i \le j)$ maps $\beta_i \setminus \Ker[i,j]$ injectively into $\beta_j$ for all $i \le j$. 
This is equivalent to decomposing $f$ as a direct sum of interval submodules (for reference see Section \ref{sec:persistence_modules}).

Suppose that the cokernel of $f(i\leq j)$ is free for each $i,j\in\{0,\ldots,m\}$ such that $i\leq j$. For each $i\leq j$, let $K_j^i$ be a complement of 
\[
\Big (\im[i-1,i] + \Ker[i, j-1] \Big)\cap\Ker[i,j]
\]
in $\Ker[i,j]$. 
That is, 
\[
\Big ( \Big (\im[i-1,i]+ \Ker[i, j-1] \Big)\cap\Ker[i,j] \Big) \oplus K_j^i=\Ker[i,j]\,.
\]
Such complements exist as a result of Theorem \ref{thm:complement_special}, Lemma \ref{lemma:intersection_closure}, and Lemma \ref{lemma:complement_shrink}.
Let $\gamma_j^i$ be a basis of $K_j^i$. 

We now present Algorithm \ref{alg:simple}, which computes an interval decomposition.

\begin{algorithm}[H]
 \caption{ Computation of an interval decomposition. }
\label{alg:simple}
\begin{algorithmic}[1]
\REQUIRE A persistence module $f:\{0\}\cup\mathbf{m}\to R\textrm{-Mod}$ that is pointwise free and finitely-generated over $R$. For each $i,j\in\{0,\ldots,m\}$ such that $i\leq j$, we require the map $f(i\leq j)$ to have a free cokernel.
\ENSURE Returns an interval decomposition
\STATE $\beta_0\leftarrow \emptyset$
\FOR{$i = 1,\ldots,m-1$}\label{loop:outer_simple}
\STATE $\beta_i\leftarrow f(i-1\leq i)(\beta_{i-1})$
\STATE $\beta_i\leftarrow \beta_i \setminus\{0\}$
\FOR{$j=i+1,\ldots,m$}\label{loop:inner_simple}
\STATE $\beta_i\leftarrow \beta_i\cup \gamma_j^i$
\ENDFOR
\ENDFOR
\RETURN $\beta_1,\ldots,\beta_{m-1}$
\end{algorithmic}
\end{algorithm}

\begin{theorem}
\label{thm:algorithm}
    Let $\beta_1,\ldots,\beta_{m-1}$ be bases that are returned by Algorithm \ref{alg:simple}. 
    Each $\beta_i$ is a basis of $f_i$ that contains a basis for each $\Ker[i,j]$ and $\im[k,i]$. Moreover, the bases $\beta_i\subseteq f_i$ yield an interval decomposition of $f$. 
\end{theorem}

\begin{proof}
We first show that each basis $\beta_i$ that Algorithm \ref{alg:simple} returns contains a basis for each $\Ker[i,j]$ and $\im[k,i]$. We proceed by induction. 

Let $i=1$. 
Because $f_0=0$, we have that $f(0\leq 1)(\beta_0)=\emptyset$. For each $j=1,\ldots,m$ in the inner loop (see line \ref{loop:inner_simple} of Algorithm \ref{alg:simple}), we append to $\beta_1$ a basis $\gamma_j^1$ of a complement $K_j^1$ of $\Ker[1,j-1]$ in $\Ker[1,j]$. 
This process gives a basis of $\Ker[1,m]=f_1$ (because $f_m=0$). 
We additionally note that $\im[0,1]=0$ and $\im[k,1]=f_1$ for $k\geq 1$.
Whether we have $\im[k,1]=0$ (when $k=0$) or $\im[k,1]=f_1$ (when $k\geq 1$), we can a find subset of $S \su \beta_1$ (either $\emptyset$ or $\beta_1$) such that $S$ spans $\im[k,1]$. 
Thus, $\beta_1$ is a basis of $f_1$ that contains a basis of every $\im[i,j]$ and $\im[k,i]$. 
This completes the base case.

Now suppose that $\beta_1,\ldots,\beta_{i-1}$ are bases of $f_1,\ldots,f_{i-1}$, respectively, that contain bases of the relevant images and kernels. We now show that $\beta_i$ is a basis of $f_i$ that contains a basis of each kernel and image. 

We first show that $f(i-1\leq i)(\beta_{i-1})\setminus\{0\}$ is linearly independent (as a multiset). 
To see this, we first note that $\beta_{i-1}$ is a basis of $f_{i-1}$ and contains a basis of each $\Ker[i-1,j]$. Consider $\kappa_{i-1}=\beta_{i-1}\cap\Ker[i-1,i]$ and $\kappa_{i-1}^\perp=\beta_{i-1}\setminus\kappa_{i-1}$. We note that $\Ker(f(i-1\leq i))=\mathrm{span}(\kappa_{i-1})$, which implies that $f(i-1\leq i)(\kappa_{i-1}^\perp)$ (which is precisely $f(i-1\leq i)(\beta_{i-1})\setminus\{0\}$) is linearly independent. 
Moreover, $f(i-1\leq i)(\beta_{i-1})\setminus\{0\}$ spans $\im[i-1,i]$, so it forms a basis of $\im[i-1,i]$

For each $j=i+1,\ldots,m$ in the inner loop (see line \ref{loop:inner_simple} of Algorithm \ref{alg:simple}), adding the elements of $\gamma_j^i$ to $\beta_i$ preserves linear independence. 
This is the case because the union of linearly independent sets of complementary spaces is linearly independent.
In fact, on the $j$th iteration, we are extending $\beta_i$ from a basis of $\im[i-1,i]+ \Ker[i, j-1]$ to a basis of $\im[i-1,i]+\Ker[i,j]$. After the $j=m$ iteration of the inner loop (line \ref{loop:inner_simple}), we have a basis of $\im[i-1,i]+\Ker[i,m]=\Ker[i,m]=f_i$ because $f_m=0$. By construction, $\beta_i$ contains a basis of every $\Ker[i,j]$.

We now show that $\beta_i$ contains a basis of every $\im[k,i]$ by induction on $i$. The base case $i=0$ is trivial. We only need to consider $k<i$ because  $\im[k,i]=f_i$ for $k\geq i$. 
The inductive hypothesis implies that $\beta_{i-1}$ contains a basis of $\im[k,i-1]$ for each $k<i$. Let $\iota_k^{i-1}\subseteq \beta_{i-1}$ be a basis of $\im[k,i-1]$. Then $f(i-1\leq i)(\iota_k^{i-1})$ is a spanning set of $f(i-1\leq i)(\im[k,i-1])=\im[k,i]$. 
Because $f(i-1\leq i)(\iota_k^{i-1})\setminus\{0\}\subseteq \beta_i$, we see that $\beta_i$ contains a spanning set (in fact, a basis, because $\beta_i$ is a basis of $f_i$) of $\im[k,i]$. Thus, $\beta_i$ contains a basis of each $\im[k,i]$, as desired. 
This concludes the argument that $\beta_i$ contains a basis of each submodule of the form $\Ker[i,j]$ and $\im[k,i]$.

It remains to show that the bases $\beta_1,\ldots,\beta_{m-1}$ decompose $f$ into a direct sum of interval submodules.
That is, we wish to show that for $i,j\in\{0,\ldots,m\}$ such that $0 < i < j < m$, the structure map $f(i \le j)$ (1) sends $\beta_i \cap \Ker[i,j]$ to 0 and (2) maps the $\beta_i \setminus \Ker[i,j]$ injectively into $\beta_j$.  
The first claim (1) holds by Definition \ref{def:Im_and_Ker}.
The second claim (2) holds because  $f(i\leq j)(\beta_i)\setminus\{0\}$ is linearly independent (as a multiset) and because $f(i\leq j)(\beta_i) \setminus \{ 0\}\subseteq \beta_j$, by construction.
\qed \end{proof}

\begin{rmk}
    Our original proof of Theorem \ref{thm:hard_direction_result} relied on the construction of ``saecular lattices,'' which are lattice submodules generated by kernels and images. 
    While this construction provided a much more general framework than which was needed to proof Theorem \ref{thm:hard_direction_result}, we have left it as an appendix section (see Appendix \ref{sec:interval_decomp_hard_old}).
    We use results from it in Section \ref{sec:matrix_algorithm}.
\end{rmk}

\section{Matrix Algorithm}
\label{sec:matrix_algorithm}

In this section, we translate Algorithm \ref{alg:simple} into the language of matrix algebra.
We describe the algorithm in Subsection \ref{sec:matrix_algorithm_presented}, prove that this algorithm is correct in Subsection \ref{sec:matrix_correct}, and provide a complexity bound in Subsection \ref{sec:matrixalgorithmcomplexity}.
We give a review of relevant facts on Smith-normal-form factorization in Appendix \ref{appendix:smith_normal_form}.

For simplicity, assume that each module $f_a$ is a copy of $R^d$, for some $d$, where $R$ is the ring of coefficients. 
We identify the elements of $R^d$ with length-$d$ column vectors, and we let $F_a$ denote the matrix such that $f(a \le a+1)(v) = F_a v$ for all $v \in f_a$.

Recall that $f_0=f_m=0$ by convention (see Section \ref{sec:uniqueness_of_interval_decompositions}). As in Section \ref{sec:algorithm}, our task is to define a sequence of bases $\beta_i \su f_i$ such that $f(i \le j)$ maps $\beta_i \setminus \Ker[i,j]$ injectively into $\beta_j$ for all $i \le j$.

We say that a matrix $S$ with entries in a PID $R$ is \emph{unimodular} (or \emph{invertible}) if there exists a matrix $S^{-1}$ with entries in $R$ such that $SS^{-1} = S^{-1}S$ is the identity.

\begin{definition}
    Let $A\in M_{r,s}(R)$ be given. We say that the \emph{span} of $A$ is the span of its columns. If $r\geq s$ and $A$ has unit elementary divisors, we say that $B\in M_{r,r-s}(R)$ is a \emph{complement} of $A$ if $[A|B]$ is unimodular.
\end{definition}

\subsection{The matrix algorithm}
\label{sec:matrix_algorithm_presented}

Here we provide a self-contained description of the matrix algorithm (see Algorithm \ref{alg:matrix}). 
The formulation of Algorithm \ref{alg:matrix} is based on that of Algorithm \ref{alg:simple}.
Additionally, the formulation of Algorithm \ref{alg:matrix} uses linear-algebra results that we discuss in Appendix \ref{appendix:smith_normal_form}, and its proof of correctness (i.e., Theorem \ref{thm:alg_matrix}) uses results about the ``saecular lattice'' that we discuss in Appendix \ref{sec:interval_decomp_hard_old}.

A \emph{kernel filtration matrix} at $f_a$ is an invertible matrix 
$$
X_a=[X_a^{a+1}| \cdots | X_a^n]
$$ 
with column submatrices $X_a^{a+1}, \ldots , X_a^n$  such that
$$
\colspace([X_a^{a+1}| \cdots | X_a^j]) = \Ker[a,j]
$$     
for all $j$.  We write $(X_a^{j})^{-1}$ for the row-submatrix of $X_a^{-1}$ such that $(X_a^{j})^{-1} X_a^j = I$ is identity (that is, 
the row indices of $(X_a^{j})^{-1}$ in $X_a^{-1}$ equal the column indices of $X_a^j$ in $X_a$).

\begin{algorithm}[H]
	\caption{Interval decomposition via matrix factorization. }
	\label{alg:matrix}
	\begin{algorithmic}[1]
		\REQUIRE A persistence module $f:\{0, \ldots, m\}\to R\textrm{-Mod}$ such that each $f_i$ equals $R^k$ for some $k$. Concretely, this module is encoded by a sequence of matrices $F_i$ representing the maps $f(i \le i+1)$ for each integer $i$ such that $1 \le i < m-1$. We do not require $f$ to admit an interval decomposition.
		\ENSURE Returns an interval decomposition of $f$, if such a decomposition exists. Otherwise returns a certificate that no such decomposition exists.
		\STATE Let $\beta_i = \emptyset$ for $i = 1, \ldots, m$
		\STATE Let $Y_0$ be the (unique and trivial) kernel filtration matrix at $f_0$
		\FOR{$i = 1,\ldots,m-2$}\label{loop:outer_matrix}
		\STATE Compute a kernel filtration matrix $X_i = [X_i^{i+1} | \cdots | X_i^m]$ at $f_i$ via repeated Smith-normal-form factorization (see Proposition \ref{prop:kernel_filtration_matrix})
		\FOR{$j = i+1,\ldots,m-1$}\label{loop:inner_matrix}
		\STATE Obtain a Smith-normal-form factorization of $A_i^j: = (X^j_i)^{-1} F_{i-1} Y_{i-1}^j$
		\IF{this factorization yields a complement $B_i^j$ of $A_i^j$ (see Proposition \ref{prop:snf_and_cokernels})}
		\STATE  Define  $Y_i^j = [X_i^j B_i^j | F_{i-1} Y^j_{i-1}]$
		\ELSE
		\STATE STOP; the persistence module $f$ does not split into interval submodules\label{stop:halt}
		\ENDIF
		\ENDFOR
		\STATE Let $\beta_i$ be the set of columns of $Y_i = [Y_i^{i+1} | \cdots | Y_i^m]$
		\ENDFOR
		\RETURN $\beta_1, \ldots, \beta_{m}$\label{stop:finish}
	\end{algorithmic}
\end{algorithm}

When matrix multiplication and the computation of a Smith normal form are polynomial-time procedures with respect to matrix dimensions, it is readily checked that Algorithm \ref{alg:matrix} is polynomial time\footnote{These procedures are known to be polynomial time when $R$ is $\Z$ or $\Q[x]$. See Remark \ref{rmk:polynomial-time_Z}.} with respect to matrix dimension and persistence-module length. 
We compute the complexity of Algorithm \ref{alg:matrix} (see Proposition \ref{prop:complexity}) relative to the complexity of matrix multiplication, matrix inversion, and computing a Smith normal form.  

\begin{rmk}
     Algorithm \ref{alg:matrix} terminates either at line 10 or line 12.
	On one hand, if Algorithm \ref{alg:matrix} terminates at line 10, then by Lemma \ref{lem:complementsrequire}, $f$ cannot decompose into a direct sum of interval modules. 
	On the other hand, if Algorithm \ref{alg:matrix} does not terminate at line 10, then it terminates at line 12.
	By Lemma \ref{lem:complementssuffice}, Algorithm \ref{alg:matrix} yields a set of bases $\beta_1, \ldots, \beta_m$ that decompose $f$ into interval modules.
\end{rmk}

\begin{rmk}
We can perform the computations of the inner-loop iterations (see line \ref{loop:inner_matrix}) of Algorithm \ref{alg:matrix} in parallel. That is, for any fixed $i\in\{1,\ldots,m-2\}$, we can compute each $A_i^j$ and its Smith-normal-form factorization for $j\in\{i+1,\ldots,m-1\}$ in parallel. 
\end{rmk}

\subsection{Correctness of Algorithm \ref{alg:matrix}}
\label{sec:matrix_correct}

We now prove that Algorithm \ref{alg:matrix} is correct.

We say that the matrices $Y_1, \dots, Y_i$ \emph{cohere} if 
$Y_k = [Y_k^{k+1}  | \cdots | Y_k^m]$ is a kernel filtration matrix at $f_k$ for all $k < i$.  By construction of $Y_k$, this immediately implies that for $k < i-1$, the following conditions hold:
\begin{itemize}
    \item the set of columns of $Y_k^{k+1}$ maps to $0$ under $F_k$;
    \item the set of columns of $Y_{k}^j$ maps injectively into the set of columns of $Y_{k+1}^{j}$ under $F_k$ for $j>k+1$.
\end{itemize}

Therefore, the bases $\beta_1, \ldots, \beta_m$, which are given by the columns of the matrices $Y_1,\ldots,Y_m$, respectively, decompose $f$ into interval modules if $Y_1, \ldots, Y_m$ cohere.

\begin{lemma}
\label{lem:complementssuffice}
Suppose that $i$ satisfies $1\leq i\leq m-1$. 
If $A_p^q$ admits a complement for every $p$ and $q$ such that $q > p$ and $p \le i$,  then $Y_1, \ldots, Y_i$ cohere.
\end{lemma}
\begin{proof}
We will show that $Y_1, \ldots, Y_i$ cohere for all $i$. We proceed by induction on $i$.  

First, suppose that $i = 1$. 
Because $Y_0$ is degenerate, each matrix $A_1^j$ is also degenerate so we can view $A_1^j$ as a matrix of size $d \times 0$ for some $d$.
It follows that any $d \times d$ unimodular matrix $B_1^j$ is a complement of $A_1^j$.
Because $X_1 = [X_1^2 | \cdots | X_1^m]$ is a kernel filtration matrix by hypothesis and the column space of $X_1^j$ equals that of $X_1^j B_1^j$, the matrix $Y_1 = [X_1^2 B_1^2 | \cdots | X_1^m B_1^m]$ is also a kernel filtration matrix.  This establishes the base case $i=1$.  

Now, suppose that the desired conclusion holds for $i < k$; consider $i = k$. 
We have that $Y_{i-1}$ is a kernel filtration matrix.  
Therefore, given a kernel filtration matrix $X_i$ at $f_i$, for all $j$ such that $i+1\leq j\leq m-1$, the columns of $F_{i-1} Y_{i-1}^j$ lie in the column space of $[X_i^{i+1} | \cdots | X_i^j]$.  
Therefore, the matrix $X_i^{-1}F_{i-1} Y_{i-1}$ has the following block structure:
\begin{align}
X_i^{-1} F_{i-1} Y_{i-1} = \begin{bmatrix}
0 & A_i^{i+1} & * & \cdots & * 
\\
0 & 0 & A_i^{i+2} & \cdots & * 
\\ \vdots & \vdots & \ddots & \ddots & \vdots \\
0 & 0 & 0 & \cdots & A_i^m\
\end{bmatrix}\,.
\label{eq_block_upper_triangular}
\end{align}
Consider the matrix
\[ {M} = 
\left[
\begin{array}{@{}ccccccccccc@{}}
A_i^{i+1} & B_i^{i+1} & * & 0 & \cdots & * & 0  \\
0 & 0 & A_i^{i+2} & B_i^{i+2} & \cdots & * & 0  \\
\vdots & \vdots & \vdots & \vdots & \ddots & \vdots & \vdots &\\
0 & 0 & 0 & 0 & \cdots & A_i^{m} & B_i^{m}\\
\end{array}
\right]\,,
\]
where $B_i^j$ is a complement of $A_i^j$. In effect, we obtain $M$ by inserting some columns into $X_i^{-1} F_{i-1} Y_{i-1}$. 
The matrix $M$ is unimodular because each matrix $[A_i^k | B_i^k]$ is unimodular.  Moreover, the product $X_iM = Y_i$ is a kernel filtration matrix because $X_i$ is a kernel filtration matrix and $M$ is block-upper-triangular with unimodular diagonal blocks.  The desired conclusion follows.
\qed \end{proof}

\begin{lemma}\label{lemma:lin_independence_A_i^j}
    Suppose $i$ satisfies $1\leq i\leq m-1$.  
    If $Y_{i-1}$ is a kernel filtration matrix, then for $j>i$, the columns of $A_i^j$ in Equation \ref{eq_block_upper_triangular} are linearly independent.
\end{lemma}
\begin{proof}
	Suppose that $j>i$, and let  
	$C = \{k+1, k+2, \ldots, k+p\}$ denote the set of column indices
	in the matrix \eqref{eq_block_upper_triangular} that contain the submatrix $A_i^j$. 
	
	Seeking a contradiction, suppose that the columns of $A_i^j$ are linearly dependent.  Then there exists a nonzero vector $\hat{v}$ such that $A_i^j \hat{v} = 0$. Define a vector $v$ by padding $\hat{v}$ with $0$ entries to align the nonzero entries with the columns that contain $A_i^j$ in the matrix \eqref{eq_block_upper_triangular}. That is, we define $v$ such that
	\[
	v_{q}=\begin{cases}
	\hat{v}_\ell\,, & {q\in C ~\textrm{and}~q=k+\ell}\\
	0\,, & \textrm{otherwise}\,.
	\end{cases}
	\]
	
	By inspection of the block structure of \eqref{eq_block_upper_triangular} and recalling that $A_i^j\hat{v}=0$, it follows that  $F_{i-1} Y_{i-1} v$  lies in the column space of $[X_i^{i+1} | \cdots | X_i^{j-1}]$, which equals $\Ker[i,j-1]$.  Consequently, $Y_{i-1}v$ lies in $\Ker[i-1, j-1]$.  However, because $Y_{i-1}$ is a kernel filtration matrix at $f_{i-1}$, the space $\Ker[i-1, j-1]$ equals the column space of $[Y_{i-1}^i | \cdots | Y_{i-1}^{j-1}]$. Therefore, the columns of $Y_{i-1}$ are linearly dependent.  This contradicts the hypothesis that the columns of $Y_{i-1}$ form a basis of $f_{i-1}$.  The desired conclusion follows. 
\qed \end{proof}

\begin{lemma}
\label{lem:complementsrequire}
    If Algorithm \ref{alg:matrix}  terminates because some $A_i^j$ lacks a complement, then $f$ does not split as a direct sum of interval modules.
\end{lemma}

\begin{proof}
	Let $i$ and $j$ be the first indices such that $A_i^j$ lacks a complement. 
	For each $p<i$, the matrix $A_p^q$ has a complement. For each $q<j$, the matrix $A_i^q$ has a complement; let $B_i^q$ be a complement of $A_i^q$.
	
	Consider the matrix 
	
	\[ {N} = 
	\left[
	\begin{array}{@{}ccccccccccc@{}}
	A_i^{i+1} & B_i^{i+1} & * & 0 & \cdots & *  \\
	0 & 0 & A_i^{i+2} & B_i^{i+2} & \cdots & *   \\
	\vdots & \vdots & \vdots & \vdots & \ddots & \vdots &  &\\
	0 & 0 & 0 & 0 & \cdots & A_i^{j}\\
	0 & 0 & 0 & 0 & \cdots & 0
	\end{array}
	\right]\,.
	\]
	
	Because each submatrix of the form $[A_i^k | B_i^k]$ is unimodular, the span of the product $X_iN$ contains the span of $[X_i^{i+1} | \cdots | X_i^{j-1}]$, which equals  $\Ker[i,j-1]$ because $X_i$ is a kernel filtration matrix. 
	
	Because every preceding matrix $A_p^q$ admits a complement, by Lemma \ref{lem:complementssuffice}, the matrix $Y_{i-1}$ is a kernel filtration matrix at $f_{i-1}$.  The columns of the matrix $[Y_{i-1}^{i} | \cdots | Y_{i-1}^j]$ therefore span the subspace $\Ker[i-1, j]$ of $f_{i-1}$.  Consequently, the columns of the product matrix $F_{i-1}[Y_{i-1}^{i} | \cdots | Y_{i-1}^j]$ span a subspace belonging to the saecular sublattice of $f_i$, because the direct image operator restricts to a homomorphism of saecular lattices (see Theorem \ref{thm:comm_diagram_forward}).  This image subspace equals the span of  $F_{i-1}[Y_{i-1}^{i+1} | \cdots | Y_{i-1}^j]$, because $F_{i-1} Y_{i-1}^i = 0$.
	Every column of $X_iN$ is a column of $F_{i-1}[Y_{i-1}^{i+1} | \cdots | Y_{i-1}^j]$ or a column in the span of $[X_i^{i+1} | \cdots | X_i^{j-1}]$, so $X_i N$ spans a subspace $S$ in the saecular lattice at $f_i$.  
	
	Because $A_i^j$ lacks a complement, at least one of two conditions must hold; either (i) $A_i^j$ has a non-unit elementary divisors or (ii) the columns of $A_i^j$ are linearly dependent.  Because $Y_{i-1}$ is a kernel filtration matrix, the columns of $A_i^j$ are linearly independent, by Lemma \ref{lemma:lin_independence_A_i^j}.  Therefore, $A_i^j$ has a non-unit elementary divisor. It follows that the quotient $f_i / S$ has torsion, which, by Lemma \ref{lemma:complement_equivalence}, implies that $S$ does not admit a complement in $f_i$. This implies that $f$ does not split as a direct sum of interval modules, by Theorems \ref{thm:problem_statement} and \ref{thm:complement_general}.
\qed \end{proof}

\begin{theorem}\label{thm:alg_matrix}
    Let $f:\{0,\ldots,m\}\to R\mathrm{-Mod}$ be a persistence module that is pointwise free and finitely-generated. 
    If $f$ has an interval decomposition, then Algorithm \ref{alg:matrix} returns an interval decomposition. 
    If $f$ has no interval decomposition, then Algorithm \ref{alg:matrix} certifies that no interval decomposition exists.
\end{theorem}
\begin{proof}
    Algorithm \ref{alg:matrix} terminates either at line 10 or at line 12.
    If $f$ has an interval decomposition, then by Lemma $\ref{lem:complementsrequire}$, Algorithm \ref{alg:matrix} does not terminate at line 10.
    This implies that Algorithm \ref{alg:matrix} terminates at line 12. 
    By Lemma \ref{lem:complementssuffice}, Algorithm \ref{alg:matrix} yields a set of bases $\beta_1, \ldots, \beta_m$ that decompose $f$ into interval modules.
    If $f$ does not have an interval decomposition, then by Lemma \ref{lem:complementssuffice}, Algorithm \ref{alg:matrix} does not terminate at line 12. So, Algorithm \ref{alg:matrix} must terminate at line 10, which certifies that $f$ does not have an interval decomposition. 
\qed \end{proof}

\begin{rmk}
	Proposition \ref{prop:complexity} provides a worst-case complexity of Algorithm \ref{alg:matrix}, but we believe that a typical computation is much faster. 
	We did not calculate the typical number (e.g., average-case complexity) of required arithmetic operations that is required by Algorithm \ref{alg:matrix}. 
\end{rmk}

\subsection{Complexity of Algorithm \ref{alg:matrix}}
\label{sec:matrixalgorithmcomplexity}

\begin{proposition}\label{prop:complexity}
    Let $f:\{0,\ldots,m\}\to R\textrm{-Mod}$ be a persistence module that is pointwise free and finitely-generated, and let $d=\max_i(\mathrm{rank}(f_i))$.
	Let $M_d$ be the complexity of matrix multiplication of two $d\times d$ matrices, $R_d$ be the complexity of inverting a $d\times d$ matrix, and $S_d$ be the complexity of computing the Smith normal form of a $d\times d$ matrix. 
	Algorithm \ref{alg:matrix} has worst-case complexity $O(m K_{m,d})$, where $K_{m,d} = \max(mM_d,mS_d,R_d)$.
\end{proposition}

\begin{proof}
	Consider the $i$th iteration of the outer loop (see line \ref{loop:outer_matrix}) of Algorithm \ref{alg:matrix}.  In this iteration, we must compute one kernel filtration matrix at $f_i$ and run once through the inner loop (see line \ref{loop:inner_matrix}). 
	
	We first account for the computation of the kernel filtration matrix at $f_i$.
	To do this, we follow the procedure in Proposition \ref{prop:kernel_filtration_matrix}. 
	For each step $k=i,\ldots, m-1$, we take $F_{k-1}\cdots F_iT_i^+\cdots T_{k-2}^+$, multiply by $F_k$ on the left, and multiply by $T_{k-1}^-$ on the right to obtain $F_k\cdots F_i T_i^+\cdots T_{k-2}^+T_{k-1}^-$. We then compute the Smith normal form. 
    Each step thus requires two matrix multiplications and one Smith-normal-form factorization. 
	Because the dimension of each $F_i$ is at most $d\times d$, the dimensions of matrices $F_k$, $T_{k-1}^+$,  $T_{k-1}^-$, and $F_{k-1}\cdots F_iT_i^+\cdots T_{k-2}^+$ is at most $d\times d$. 
	Lastly, we need to compute $F_k\cdots F_i T_i^+\cdots T_{k-2}^+T_{k-1}^+$, which is needed for step $k+1$. 
	This requires an additional matrix multiplication, where we multiply by $T_{k-1}^+$ instead of by $T_{k-1}^-$ on the right. 
    In total, we require the computation of three matrix multiplications and one Smith normal form.
    The total computational cost is thus at most $3M_d + S_d$. 
	We iteratively perform these operations $m-i$ times, so in total the computation of a kernel filtration matrix requires at most $(m-i)(3M_d + S_d)$ arithmetic operations.

	We now account for the computations of the inner loop (see line \ref{loop:inner_matrix} of Algorithm \ref{alg:matrix}).  
	By construction, each matrix $A_i^j$ is a submatrix of $X_i^{-1} F_{i-1} Y_{i-1}$. 
	Thus, the combined cost of computing every matrix $A_i^j$ is at most the cost of one matrix inversion and two matrix multiplications (of matrices with sizes of at most  $d \times d$). 
	In addition, for each $j=i+1, \ldots, m-1$, we must compute (1) a complement $B_i^j$ of $A_i^j$ (if one exists), which  requires the computation of one Smith normal form, and (2) a matrix $Y_i^j$, which requires two matrix multiplications (see line 8 of Algorithm \ref{alg:matrix}). Thus, in total, we require one matrix inversion, $2 + 2(m-i-1) = 2(m-i)$ matrix multiplications, and $m-i-1$ Smith-normal-form factorizations.
    In total, the inner loop computations requires at most $2(m-i)M_d+(m-i-1)S_d+R_d$ arithmetic operations.
	
	Overall, each iteration $i=1,\ldots,m-2$ of the outer loop (see line \ref{loop:outer_matrix}) of Algorithm \ref{alg:matrix} requires at most of at most $5(m-i)M_d+(2m-2i-1)S_d+R_d$ arithmetic operations.
	The worst-case complexity of Algorithm \ref{alg:matrix} therefore is $O(m K_{m,d})$, where $K_{m,d} = \max(mM_d,mS_d,R_d)$.
\qed \end{proof}

\section{Summary of complexity results}
\label{sec:complexity_Summary}

In this work we have considered Problems \ref{prob:splits} and \ref{prob:int_decomp} (which we restate for convenience) concerning a pointwise free and finitely-generated persistence module $f:\{0,\ldots,m\}\to R\textrm{-Mod}$ over a PID $R$. Recall that $d=\max_i(\textrm{rank}(f_i))$ denotes the maximum of the ranks of the $R$-modules in  our persistence module, $S_d$ denotes the complexity of computing Smith normal form of a $d\times d$ matrix, $M_d$ denotes the complexity of computing matrix multiplication of two $d\times d$ matrices, and $R_d$ denotes the complexity of inverting a unimodular $d\times d$ matrix over $R$. 

{
\renewcommand{\theproblem}{\ref{prob:splits}} 
\begin{problem}
    Determine whether $f$ splits as a direct sum of interval submodules.
\end{problem}
\addtocounter{problem}{-1} 
}

{
\renewcommand{\theproblem}{\ref{prob:int_decomp}} 
\begin{problem}
    If an interval decomposition does exist, then compute one explicitly.
\end{problem}
\addtocounter{problem}{-1} 
}

Our work addresses Problem \ref{prob:splits} by relating the existence of an interval decomposition to the cokernels of the structure maps of $f$.
Algorithmically, this translates to repeatedly performing Smith normal forms of the matrix representations of the structure maps. Theorem \ref{thm:cokernelcomplexity} gives the complexity of this procedure. To the best of our knowledge, this is the first finite-time (respectively, polynomial-time)\footnote{Concretely, the procedure is finite-time (respectively, polynomial-time) if the problem of computing Smith normal form of a matrix over $R$ is finite-time (respectively, polynomial-time).} solution to Problem \ref{prob:splits}  which is compatible with every pointwise free and finitely-generated persistence module over a PID.

\begin{rmk}\label{rmk:polynomial-time_Z}
    When the coefficient ring $R$ is $\Z$ or $\Q[x]$, matrix multiplication and computing Smith normal form are polynomial-time operations (see \cite{storjohann_near_1996,storjohann_fast_1997}). 
    This implies that when considering these coefficient rings, Algorithm \ref{alg:matrix} is also polynomial time.
\end{rmk}

\begin{theorem}
\label{thm:cokernelcomplexity}
Suppose the matrix representations for all of the structure maps are already computed. One can solve Problem \ref{prob:splits} in $O(m^2S_d)$. 
\end{theorem}

\begin{proof}
By Theorem \ref{thm:problem_statement}, to check whether a persistence module $f$ splits, it is enough to check that each $f(i\leq j)$ has free cokernel. 

Let $M_i^j$ be the matrix representation of the structure map $f(i\leq j)$. It is enough to check that each $M_i^j$ only consists of unit elementary divisors. One can do this by computing the Smith normal for of $M_i^j$. Therefore, because there are $m(m-1)$ structure maps to check (recall that $f_0=f_m=0$), this procedure requires $m(m-1)$ Smith normal form computations, and therefore has computational complexity $O(m^2S_d)$.
\qed \end{proof}

To the best of our knowledge, Algorithm \ref{alg:matrix} is the first finite-time solution to Problem \ref{prob:int_decomp} (as well as Problem \ref{prob:splits}) which is compatible with every pointwise free and finitely-generated persistence module over a PID.

\begin{theorem}
    Algorithm \ref{alg:matrix} solves  Problem \ref{prob:int_decomp} (in addition to Problem \ref{prob:splits}) for pointwise free and finitely-generated persistence modules of abelian groups,  with complexity $O(m\cdot\max(mM_d,mS_d,R_d))$.
\end{theorem}

\begin{proof}
By Theorem \ref{thm:alg_matrix}, Algorithm \ref{alg:matrix} solves Problems \ref{prob:splits} and \ref{prob:int_decomp}. 
By Proposition \ref{prop:complexity}, Algorithm \ref{alg:matrix} has time complexity $O(m\cdot\max(mM_d,mS_d,R_d))$.
\qed \end{proof}

Obayashi and Yoshiwaki \cite{obayashi_field_2023} provide additional structural insight into Problem \ref{prob:splits}.  
Their work provides a homological condition (namely, that relative homology groups with $\Z$-coefficients are free) under which, given a filtration, the corresponding persistence diagram is independent of the choice of field. 
The authors do not directly relate this condition to a splitting of the corresponding persistence module.
However, by Theorem \ref{thm:OurFieldIndependence}, their condition on the freeness of the relative homology groups is equivalent to the condition that all structure maps of the corresponding persistence module have free cokernels.
Therefore, Obayashi and Yoshikawa's condition on the relative homology groups of a filtration is equivalent to the corresponding persistence module admits an interval decomposition.

Obayashi and Yoshiwaki additionally provide\footnote{This result was not formally refereed, however their proof is correct and we provide an independent argument in Appendix \ref{sec:standardalgorithm}.} a solution to Problems \ref{prob:splits} and \ref{prob:int_decomp}, in the special case of persistence modules realized via integer persistent homology of simplicial complexes with simplex-wise filtrations\footnote{Recall that a \emph{simplex-wise filtration} on a simplicial complex $K = \{\sigma_1, \ldots,\sigma_m\}$ is a nested sequence of sub-simplicial complexes $\emptyset = \mathcal{K}_0 \subseteq \cdots \subseteq \mathcal{K}_m = K$ such that $\mathcal{K}_p = \{\sigma_1, \ldots, \sigma_p\}$ for all $0 \le p \le m\}$.}. The procedure consists of performing a standard algorithm to compute persistent homology for simplex-wise filtrations, with an additional break condition.  
This method does not obviously adapt to arbitrary persistence modules. 
However, when their algorithm produces an interval decomposition, the complexity bound is much better. See Appendix \ref{sec:standardalgorithm} for further discussion.

\begin{theorem}\label{thm:obayashi_complexity}
    Suppose that $f = H_k(\mathcal{K};\Z)$ is the  persistence module of a simplex-wise filtration on a finite simplicial complex $\mathcal{K}=\{\mathcal{K}_i\}_{i=0}^{m-1}$ (by convention, we suppose that
    $H_k(\mathcal{K}_m;\Z)=0$; note that we can assume this without loss of generality). Then we can solve Problems \ref{prob:splits} and \ref{prob:int_decomp} by applying the standard algorithm (Algorithm \ref{alg:standard} in Appendix \ref{sec:standardalgorithm}) to the associated boundary matrix, with rows and columns arranged in birth order.  This procedure has complexity $O(m^3)$.
\end{theorem}
\begin{proof}
    See Theorems \ref{thm:yoshiwakiAlgorithmCorrect} and \ref{thm:standardalgorithmcomplexity} in Appendix \ref{sec:standardalgorithm}.
\qed \end{proof}

\begin{rmk}
In practice, the standard algorithm has often been observed to run in linear time \cite{bauer2017phat}.     
\end{rmk}

\begin{rmk}
In fact, we can do better. The result of Obayashi and Yoshiwaki depends only on the pair of matrices $R,V$ produced by the standard algorithm. These matrices can be computed in $O(m^\theta)$ time \cite{morozov2024persistentcohomologymatrixmultiplication}, where $m^\theta$ is the asymptotic complexity of multiplying two matrices of size $m \times m$, with real coefficients.
\end{rmk}

\begin{rmk}
In Appendix \ref{sec:standardalgorithm} we generalize Theorem \ref{thm:obayashi_complexity} in two directions. First, we can solve Problems \ref{prob:splits} and \ref{prob:int_decomp} by applying the standard algorithm to the \emph{anti-transpose} of the boundary matrix. This has the same worst-case complexity; however the decomposition of anti-transposed matrices has been observed to run significantly faster on many types of scientific data \cite{de2011dualities, bauer2017phat}. Second, we show that the same procedure can be applied to any filtered chain complex which meets certain requirements; for example, the filtered chain complex of a cubical or CW complex in which we add only one cell to the filtration at a time.
\end{rmk}

In closing, we compare the asymptotic complexity bounds of determining field independence of a non-simplex-wise filtration using the criterion of Obayashi-Yoshiwaki versus the cokernel condition presented in this paper.
Suppose that $f = H_k(\mathcal{K};\Z)$ is the  persistence module of a filtration on a finite simplicial complex $\mathcal{K}=\{\mathcal{K}_i\}_{i=0}^{m-1}$ (by convention, we define $H_k(\mathcal{K}_m;\Z)=0$), but remove the condition that this filtration be simplex-wise.  Let $e$ denote the number of $(k-1)$-, $k$-, or $(k+1)$-simplices, whichever is greatest.

\begin{theorem}\label{thm:comparative_freeness_complexity}
     Under these conditions, we can solve Problem \ref{prob:splits} either by computing and checking the freeness of $H_k(\mathcal{K}_b,\mathcal{K}_a;\mathbb{Z})$ for all $0 \le a \le b \le m$, with worst-case time complexity $O(m^2 S_e)$, or by first constructing matrix representations of the module $H_k(\mathcal{K};\Z)$ with complexity $O(m S_e)$, then checking that the cokernel of every structure map $f_a \to f_b$ is free, with time complexity $O(m^2 S_d)$.
\end{theorem}

\begin{proof}
    The relative homology group $H_k(\K_b,K_a;\Z)$  is free if and only if the elementary divisors of the relative boundary matrix $\tilde \partial_{k+1}$ (which is obtained by deleting rows and columns of the ordinary boundary matrix) are units. This condition can be checked by computing the Smith Normal form of $\tilde \partial_{k+1}$ with $S_e$ operations. 
    There are $\binom{m}{2}$ pairs $(a,b)$ of indices values to check, leading to a complexity bound of $O(m^2 S_e)$ for the solution method which computes freeness of each relative homology group.

    To compute bases for the homology groups $H_k(\mathcal{K}_a;\Z)$, and matrix representations for the corresponding induced maps $H_k(\mathcal{K}_a;\Z) \to H_k(\mathcal{K}_{a+1};\Z)$, one can employ a standard sequence of Smith Normal factorizations. These operate on the matrices of $\partial_{k+1}$, $\partial_k$, or variants of these matrices of equal or lesser size obtained by deletions and basis changes. The complexity of each factorization is $O(S_e)$, and we must compute the homology groups of $m+1$ spaces and $m$ maps. Thus the operations required are $O(m S_e)$.  To check that the cokernel of each map $f_a \to f_b$ is free, one need only check the elementary divisors by placing the corresponding matrix in Smith Normal form; this requires at most $\binom{m}{2} S_e = O(m^2 S_e)$ operations.
\qed \end{proof}

In real-world homology calculations, the maximum value $d$, of $\mathrm{rank}(H_k(\mathcal{K}_i; \Z))$ over all $i$, is much smaller than $e$.
For example, the value of $d$ (with respect to first homology) for the Vietoris--Rips complex of a point cloud of 1,000 points uniformly sampled from the unit cube in $\mathbb{R}^3$ is a few hundred\footnote{In analyzing 1,000 point clouds --- each with 1,000 points generated uniformly at random on the unit cube --- we observed values of $d$ ranging from $158$ to $235$, with a median value of $193$.} while $e$ is on the order of $\binom{1000}{3}$. 
Under these conditions, the asymptotic complexity bound $O(m S_e + m^2 S_d)$ will be much tighter than $O(m^2 S_e)$.

\section{Conclusion}
\label{sec:conclusion}

We showed that a finitely-indexed persistence module that is pointwise free and finitely-generated over a PID splits as a direct sum of interval submodules if and only if the cokernel of every structure map is free (see Theorem \ref{thm:problem_statement}).  
The necessity direction of this equivalence is simple to prove (see Section \ref{sec:interval_decomp}), 
while the sufficiency direction (see Section \ref{sec:interval_decomp_hard}) required more work.
In particular, we showed the sufficiency direction by formulating an algorithm (see Algorithm \ref{alg:simple}) that takes in a persistence module and inductively constructs a sequence of bases that together comprise an interval decomposition.

We gave a concise algorithm that returns either an interval decomposition or a certificate that no such decomposition exists. This algorithm has two variants. 
The first, Algorithm \ref{alg:simple} (which we used to prove the sufficiency direction of Theorem \ref{thm:problem_statement}), is primarily combinatorial; it exposes the underlying algebra in a simple format. It works by iteratively 
constructing bases by extending the image of bases at the previous points. That is, we construct a basis $\beta_i$ of $f_i$ by extending the image $f(i-1\leq i)(\beta_{i-1})$ of a basis $\beta_{i-1}$ of $f_{i-1}$.
The second, Algorithm \ref{alg:matrix},  translates this procedure into the language of matrix algebra; it requires no special machinery beyond Smith normal form, and it offers natural possibilities for parallelization.

\subsection{Future directions}

The algorithms presented in this work are coarse in nature, and we are confident that Algorithm \ref{alg:matrix}, in particular, can be improved. 
The complexity bounds placed on Algorithm \ref{alg:matrix} can also be refined, such as by accounting for the distribution of the sizes of the matrices involved (rather than taking the max), as well as sparsity and other factors. 
Similar refinements could also be applied to the analysis of several other procedures described in this work, such as the procedure to check field-independence by computing the Smith Normal Form of a quadratic number of relative boundary matrices (see Section \ref{sec:complexity_Summary}). In practice, all such estimates should be complemented by numerical experiments on real and simulated data.
While implementation and numerical simulations fall outside the scope of this paper, they are natural and necessary next steps. 
The work presented here primarily serves to motivate and frame such work.

Another direction to explore is to improve upon the standard algorithm (Algorithm \ref{alg:standard}, see Appendix \ref{sec:standardalgorithm}). Currently, the standard algorithm only works for persistence modules that can be realized by the persistent homology of a simplex-wise filtration (or the natural analog of a simplex-wise filtration for more general families of chain complexes, see Appendix \ref{sec:standardalgorithm}). 
It would be interesting to extend it to the setting of arbitrary filtered chain complexes, while maintaining a comparable computational complexity. This direction seems tractable, in principle, because the left-to-right clearing operations of the standard algorithm generalize naturally to block-clearing operations, where blocks correspond to level sets of the filter function. 
It is also computationally attractive, as it amounts to a Smith Normal factorization of a single boundary matrix; by comparison, the only current alternative of which we are aware is the method described in Section \ref{sec:universal_cycle_reps}, whose first step alone requires $O(m)$ Smith Normal factorizations to construct matrix representations for the structure maps in the associated integer persistent homology module.

An additional theoretical direction is to extend our work to the setting of persistence modules of general $R$-modules, rather than restricting ourselves to free modules. 
This requires us to generalize the notion of ``interval decomposition'' to the setting of persistence modules valued over modules with torsion.
One can then determine the conditions on persistence modules for the existence of interval decompositions. One can perhaps take inspiration from \cite{patel_generalized_2018} for insights into generalized persistence.

It also would be interesting to
generalize our work from the setting of finitely-indexed persistence modules (for which the domain category is a finite totally-ordered poset category) to continuously-indexed persistence modules (for which the domain category is a continuous-valued totally-ordered poset category, such as $[0,1]$ or $\R$). We expect that many of the ideas from the present paper can be adapted to this more general setting, although the treatment may require extra care with respect to limits (specifically, upper-continuity of the saecular lattice; see \cite{ghrist_saecular_2021}). 

It is natural to ask whether the methods explored in this work can also be applied in the setting of zig-zag persistence modules. A persistence module can be described as consisting of a sequence $\{M_i\}_{i=1}^r$ of $R$-modules with maps $\phi_{i}:M_i\to M_{i+1}$. Zig-zag persistence modules \cite{carlsson_zigzag_2010} are generalizations of persistence modules in that they consist of a  sequence $\{M_i\}_{i=0}^m$ of $R$-modules, but the maps $\phi_i$ can go in either direction (i.e., either $\phi_i:M_i\to M_{i+1}$ or $\phi_i:M_{i+1}\to M_i$). When the ring of coefficients is a field, Gabriel's theorem \cite{gabriel_unzerlegbare_1972} guarantees that any zig-zag persistence module admits an interval decomposition. It would be interesting to consider the conditions under which zig-zag persistence modules of free and finitely-generated modules over a PID admit interval decompositions.

\appendix

\section{An Alternate Proof of Sufficiency for Theorem \ref{thm:problem_statement}}
\label{sec:interval_decomp_hard_old}

We now provide an alternate framework to prove sufficiency of Theorem \ref{thm:problem_statement}, which we repeat here for ease of reference.
We will make use of some notation and results from Section \ref{sec:interval_decomp_hard}.

{
\renewcommand{\thetheorem}{\ref{thm:hard_direction_result}}
\begin{theorem}[Sufficiency]
    Let $f$ be a persistence module that is pointwise free and finitely-generated over $R$, and suppose that the cokernel of every structure map $f(a\leq b)$ is free. Then $f$ splits into a direct sum of interval modules.
\end{theorem}
\addtocounter{theorem}{-1}
}

Our approach will be to build an interval decomposition, constructively. 
To do so, we begin by studying several submodules which can be obtained from the images and kernels of the maps $f(i \le j)$ via sum and intersection. 
We will consider a certain family of nested pairs of these submodules, $ A \subseteq B $, and show that if cokernels are free, then every pair of this form admits a complement (i.e. $B = A \oplus C$). The family of complements $C$ can then be stitched together to form an interval decomposition.
There are two challenges in this approach which must be overcome. The first is establishing existence of the complements $C$, and the second is showing that the complements present in each $f_a$ actually decompose $f_a$ as a direct sum. To overcome these challenges, we enlarge the family of submodules to a larger set, which forms a sublattice of the submodule lattice of each $f_a$. This sublattice has structural properties which render both challenges feasible.

Before building the necessary machinery for the proof of Theorem \ref{thm:hard_direction_result}, we first provide intuition by illustrating the key ideas of our proof. See Figures \ref{fig:image_kernel} and \ref{fig:map} for illustrations. 

For each $f_a$, we will construct a family $\{A^{ij}_a\}_{1\leq i,j\leq m}$ of submodules such that 
\begin{align*}
    \bigoplus_{j\leq y}A^{ij}_a &= \Ker[a,y]\,, \\
    \bigoplus_{i\leq x}A^{ij}_a &= \im[x,a]\,.
\end{align*}

See Figure \ref{fig:image_kernel} for an illustration.

\begin{figure}
    \centering
    \subfloat{\includegraphics[width=0.45\textwidth, trim= 0.62in 1.4in 6.75in 1in,clip]{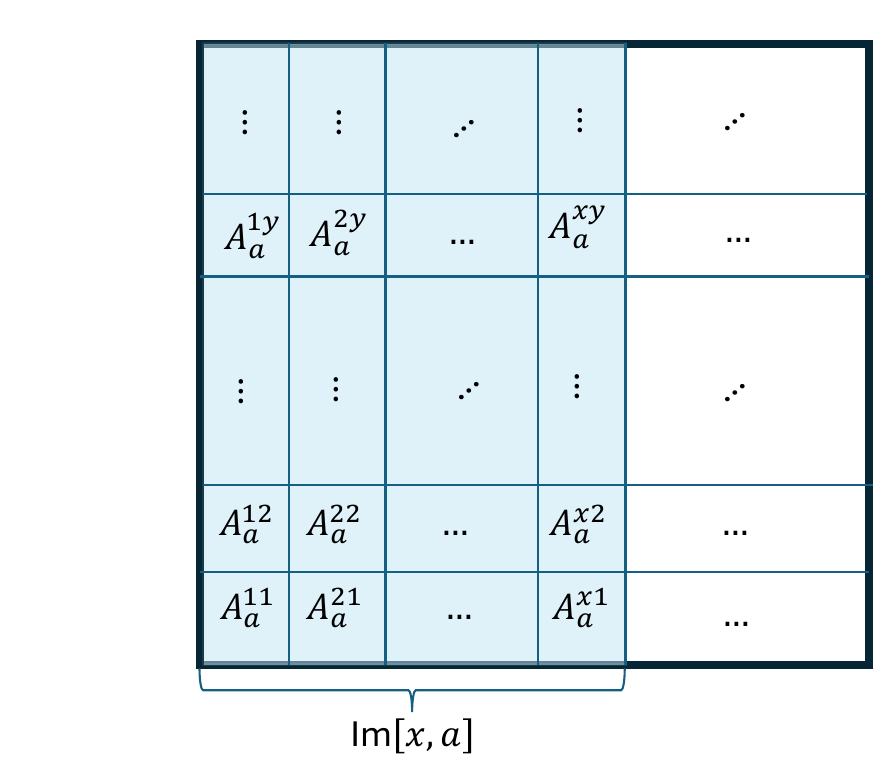}}
    \hspace{1cm}
    \subfloat{\includegraphics[width=0.45\columnwidth, trim= 0.62in 1.4in 6.75in 1in,clip]{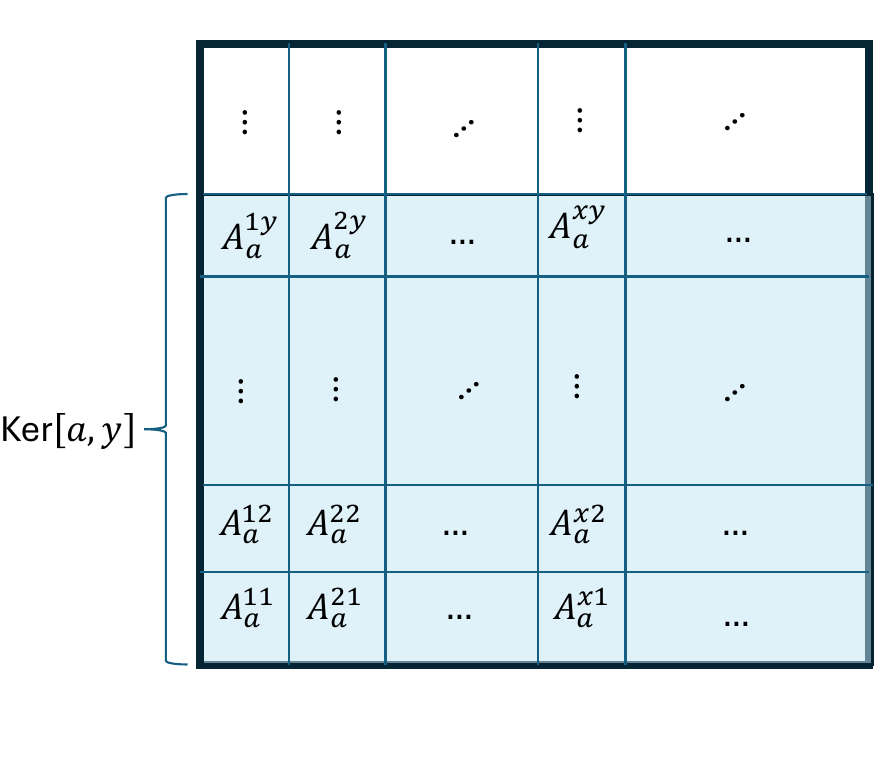}}
    \caption{Given $f_a$, we illustrate how the summands given by $\{A^{ij}_a\}_{1\leq i,j\leq m}$ form (left) $\im[x,a]$ and (right) $\Ker[a,y]$.}
    \label{fig:image_kernel}
\end{figure}

In the context of persistent homology, one can view $A^{ij}_a$ as a submodule of $f_a$ consisting of  homology classes  that are born at $i$ and die at $j$. Note that $A^{ij}_a=0$ when $a\notin[i,j)$.
These can be thought of as the ``pair groups'' discussed in \cite{cohen-steiner_extending_2009}.
We construct our family of submodules $\{A^{ij}_a\}_{1\leq i,j\leq m}$ for each $f_a$ such that when fixing $a,b\in\{0,\ldots,m\}$ with $a\leq b$, we have \footnote{See Figure \ref{fig:map} for an illustration}
\[
f(a\leq b)(A^{ij}_a)\subseteq A^{ij}_b\,.
\]
In this setting, the restricted map $f(a\leq b)|_{A^{ij}_a}:A^{ij}_a \to  A^{ij}_b$ is an isomorphism when $a,b\in[i,j)$.  
If $a$ or $b$ lies outside $[i,j)$, then  $A^{ij}_a$ or $A^{ij}_b$ is $0$, so the map between them is a zero map.

\begin{figure}
    \centering
    
    \includegraphics[width = \textwidth, trim = 1in 1in 0.7in 1in,clip]{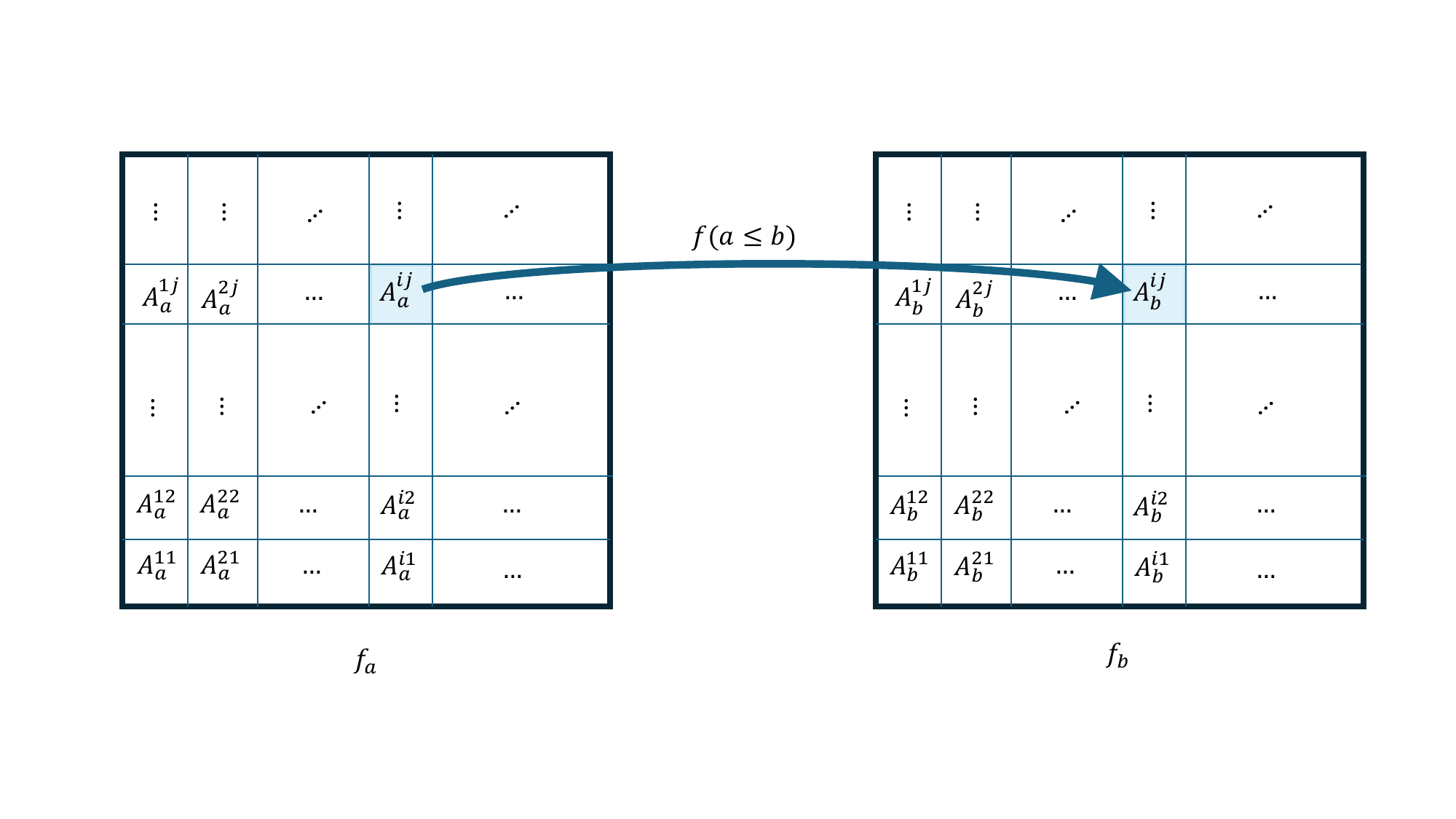}
    \caption{Given $a$ and $b$ with $a\leq b$, we illustrate that $f(a\leq b)(A^{ij}_a)\subseteq A^{ij}_b$.}
    \label{fig:map}
\end{figure}

This construction allows us to identify $A^{ij}_a$ with $A^{ij}_b$ whenever both of these modules are nonzero. 
Because a basis of $A^{ij}_a$ yields a basis of $A^{ij}_b$ (and vice versa), we can start with an appropriate basis on the left end of our persistence module and work our way right. 
That is, we can use a basis of $f_a$ to construct a suitable basis of $f_b$.
This yields a collection of bases (one for each $f_a$) that forms an interval decomposition.

\subsection{Saecular submodule lattices and homomorphisms} 

We write $\mathrm{Sub}(f_a)$ for the order lattice of submodules of $f_a$, where the meet and join operations are intersection and sum, respectively.

\begin{definition}\label{def:saecular_lattice}
  The \emph{saecular lattice} of $f_a$ is the sublattice of $\mathrm{Sub}(f_a)$ generated by all submodules of the form $\im[x,a]$ and $\Ker[a,y]$.  We denote this sublattice by $\ImKer(f_a)$.
\end{definition}

Given $a\leq b$ we can relate submodules in $\ImKer(f_a)$ and $\ImKer(f_b)$ by pushing forward (i.e., applying $f(a\leq b)$ to submodules in $\ImKer(f_a)$) and pulling back (i.e., applying $f(a\leq b)^{-1}$ to submodules in $\ImKer(f_b)$). Lemmas \ref{lemma:basic_annoying_identities}, \ref{lemma:second_annoying_lemma}, and \ref{lemma:annoying_lemma_backwards} precisely express these relationship. 

\begin{lemma}\label{lemma:basic_annoying_identities} Fix $a,b\in\mathbf{m}$ such that $a\leq b$. The following hold for any $x,y \in \mathbf{m}\cup\{0\}$: 
\begin{align}
f(a \leq b)(\im[x,a])&= \im[x,b]\cap \im[a,b]\,, 
\label{eq_pushim} \\
f(a \leq b)(\Ker[a,y])&= \Ker[b,y]\cap \im[a,b]\,.
\label{eq_pushker}
\end{align}

\end{lemma}

\begin{proof}
We first show identity \eqref{eq_pushim}.
Let $L$ and $R$, respectively, denote the left-hand and right-hand sides of the identity \eqref{eq_pushim}.  

Suppose, first, that $x \le a$.  Then $\im[x,b] \subseteq \im[a,b]$, so 
\begin{align*}
    R 
    &= \im[x,b] \\
    &= f(x \le b)(f_x) \\
    &= (f(a \leq b) \circ f(x \le a))( f_x) \\
    & = f(a \leq b)(\im[x,a]) \\
    &=L\,.
\end{align*}

Now suppose that $a \leq x \leq b$. Then $\im[x,a]=f_a$, so  $L=f(a \leq b)(f_a)=\im[a,b]$.  Because $\im[a,b]\subseteq \im[x,b]$,  the right-hand side is also $\im[a,b]$. 
	
Finally, suppose that $a \leq b \leq x$. Then the left-hand side is again $f(a \leq b)(f_a)=\im[a,b]$. Because $\im[x,b]=f_b$, the right-hand side is also $\im[a,b]$.

We now show identity \eqref{eq_pushker}. Let $L$ and $R$, respectively, denote the left-hand and right-hand sides of identity \eqref{eq_pushker}. 

Suppose first that $y \leq a \leq b$. Then  $\Ker[a,y]$ and $\Ker[b,y]$ are $0$, by definition, and $0=L=R$.

Now suppose that $a \leq y \leq b$.  Then $\Ker[a,y] \subseteq \Ker[a,b]$, so $0=f(a \leq b)(\Ker[a,y]) = L$.  The right-hand side of \eqref{eq_pushker} also vanishes, because $\Ker[b,y] = 0$, by definition.

Finally, suppose $a  \leq b \leq y$.  Then $L \subseteq \Ker[b,y]$, because $f(b \le y)( f(a \le b)(\Ker[a,y])) = f(a \le y)( \Ker[a,y]) = 0$.  Because $L \subseteq \im[a,b]$ as well, it follows that $L \subseteq R$.  
To prove the opposite containment, suppose that $z'\in L$. Then $z'=f(a \leq b)(z)$ for some $z\in f_a$, which implies  $0 = f(b \leq y)\circ f(a \leq b)(z)=f(a \leq y)(z)$.  Thus, $z\in \Ker[a , y]$, so $z' \in f(a\leq b)( \Ker[a,y]) = L$.  Because $z'$ is arbitrary, it follows that $R \su L$.  The desired conclusion follows.
\qed \end{proof}

\begin{lemma}
\label{lemma:second_annoying_lemma}
Fix $a,b\in\mathbf{m}$ such that $a\leq b$. The following holds for any $x,y \in \mathbf{m}\cup\{0\}$: 
\[
f(a\leq b) \left(\mathrm{Im}[x,a]\cap \Ker[a,y] \right)=\mathrm{Im}[x,b]\cap \Ker[b,y]\cap \mathrm{Im}[a,b]\,.
\]
\end{lemma}

\begin{proof}
First, consider $a\leq x$. In this case, $\mathrm{Im}[x,a]=f_a$ and $\mathrm{Im}[a,b]\subseteq\mathrm{Im}[x,b]$. Our task therefore reduces to showing that $f(a\leq b)(\Ker[a,y]) = \Ker[b,y]\cap \mathrm{Im}[a,b]$, which we proved in Lemma \ref{lemma:basic_annoying_identities}. 

Now suppose that $y\leq b$. Then $f(a\leq b)(\Ker[a,y])=0$, so  $f(a\leq b)(\mathrm{Im}[x,a]\cap \Ker[a,y])=0$.  Therefore, the left-hand side vanishes. Because $\Ker[b,y]=0$, by definition, the right-hand side also vanishes. This proves the desired result for $y\leq b$.

This leaves $x<a\leq b< y$. 
Because $x<a$, we have $\mathrm{Im}[x,b]\subseteq\mathrm{Im}[a,b]$, so this reduces to showing that
$$
f(a\leq b)(\mathrm{Im}[x,a]\cap \Ker[a,y])=\mathrm{Im}[x,b]\cap \Ker[b,y]\,.
$$
Applying $f(a\leq b)$ to each term in the intersection on the left-hand side yields a subset of the corresponding term in the intersection on the right-hand side, so the left-hand side is a subset of the right-hand side. 
For the other inclusion, fix an arbitrary element of $\im[x,b]\cap \Ker[b,y]$, which can be expressed in the form $f(x\leq b)(\alpha)$ for some $\alpha$. 
We note that $f(x\leq b)(\alpha)=f(a\leq b)\circ f(x\leq a)(\alpha)$, so it is sufficient to show that $f(x\leq a)(\alpha)\in \Ker[a,y]$. This does indeed hold, because $f(a\leq y)\circ f(x\leq a)(\alpha)=f(b\leq y)\circ f(a\leq b)\circ f(x\leq a)(\alpha)=0$. Thus, we see that $f(x\leq a)(\alpha)\in \mathrm{Im}[x,a]\cap\Ker[a,y]$, so $f(x\leq b)(\alpha)\in f(x\leq b)(\mathrm{Im}[x,a]\cap\Ker[a,y])$. 
\qed \end{proof}

\begin{lemma}\label{lemma:annoying_lemma_backwards}
Fix $a,b\in\mathbf{m}$ such that $a\leq b$. The following holds for any $x,y \in \mathbf{m}\cup\{0\}$: 
\begin{equation}\label{eq:annoying_lemma_background}
    f(a\leq b)^{-1}(\mathrm{Im}[x,b]+\Ker[b,y])=\mathrm{Im}[x,a]+\Ker[a,y]+\Ker[a,b]\,.
\end{equation}

\end{lemma}

\begin{proof}
Throughout this proof, let $L$ and $R$ denote the left-hand side and right-hand side of Equation \eqref{eq:annoying_lemma_background}, respectively.

We consider three cases: (i) $a\leq x$, (ii) $y\leq b$, and (iii) $x\leq a\leq b\leq y$. 

We first consider (i). In this case, we see that the right-hand side is simply $f_a$, because $\im[x,a]=f_a$. Now, if $b\leq x$, then $\im[x,b]=f_b$, which implies that the left-hand side is $f(a\leq b)^{-1}(\mathrm{Im}[x,b]+\Ker[b,y])=f(a\leq b)^{-1}(f_b)=f_a$, as desired. 
Otherwise, with $a\leq x\leq b$, we see that 
\begin{align*}
f(a\leq b)^{-1}(\mathrm{Im}[x,b]+\Ker[b,y]) &\supseteq f(a\leq b)^{-1}(\mathrm{Im}[x,b]) \\ 
&\supseteq f(a\leq b)^{-1}(\mathrm{Im}[a,b]) \\
&= f_a
\end{align*}
so the left-hand side is $f_a$ as well.

Now consider (ii). In this case, $\Ker[b,y]=0$ and $\Ker[a,y]=0$. 
Therefore, we only need to show that
\[
f(a\leq b)^{-1}(\mathrm{Im}[x,b]) = \mathrm{Im}[x,a]+\Ker[a,b]\,.
\]
We subdivide (ii) into three subcases: $b\leq x$, $a\leq x\leq b$, and $x\leq a$.

If $b\leq x$, we see that $\mathrm{Im}[x,b]=f_b$ and $\mathrm{Im}[x,a]=f_a$; because $f(a\leq b)^{-1}(f_b)=f_a$, the desired expression holds. 

If $a\leq x\leq b$, then the right-hand side is $f_a$ because $\mathrm{Im}[x,a]=f_a$. We note that 
\begin{align*}
f(a\leq b)^{-1}(\mathrm{Im}[x,b]) &\supseteq f(a\leq b)^{-1}(\mathrm{Im}[a,b]) \\
&=f_a\,,
\end{align*}
so the left-hand side is $f_a$, and the desired expression again holds.

Consider $x\leq a$. First, we show that $L \supseteq R$. If we take $f(x\leq a)(w)\in \mathrm{Im}[x,a]$, then $f(a\leq b)\circ f(x\leq a)(w)=f(x\leq b)(w)\in \mathrm{Im}[x,b]$. We see that $\Ker[a,b]\subseteq f(a\leq b)^{-1}(\mathrm{Im}[x,b])$, so $L\supseteq R$. To see that $L \subseteq R$, we note that if $z\in f(a\leq b)^{-1}(\mathrm{Im}[x,b])$, then $f(a\leq b)(z)=f(x\leq b)(\alpha)$ for some $\alpha\in f_x$. Because $f(x\leq b)=f(a\leq b)\circ f(x\leq a)$, we see that $f(a\leq b)(z)=f(a\leq b)\circ f(x\leq a)(\alpha)$, which implies that $z = f(x\leq a)(\alpha)+\zeta$ for $\zeta\in\Ker[a,b]$, as desired. 

Finally, we consider (iii). In this situation, $\Ker[a,b]\subseteq\Ker[a,y]$, so our desired expression reduces to 
\[
f(a\leq b)^{-1}(\mathrm{Im}[x,b]+\Ker[b,y])=\mathrm{Im}[x,a]+\Ker[a,y]\,.
\]
We first show that $L \supseteq R$. Indeed, if we take $f(x\leq a)(w)\in \mathrm{Im}[x,a]$, then $f(a\leq b)\circ f(x\leq a)(w) = f(x\leq b)(w)\in \mathrm{Im}[x,b]$. Additionally, if we take $z\in\Ker[a,y]$, then $f(a\leq b)(z)\in\Ker[b,y]$, as $f(b\leq y)\circ f(a\leq b)(z)=f(a\leq y)(z)=0$. Thus, $L \supseteq R$. 

To see that $L \subseteq R$, take $z\in f(a\leq b)^{-1}(\mathrm{Im}[x,b]+\Ker[b,y])$. We note that $f(a\leq b)(z) = f(x\leq b)(\alpha)+\zeta$ for $\alpha\in f_x$ and $\zeta\in\Ker[b,y]$. As $f(x\leq b) = f(a\leq b)\circ f(x\leq a)$, the right-hand side is equal to $f(a\leq b)\circ f(x\leq a)(\alpha)+\zeta$. Applying $f(b\leq y)$ to both sides then gives $f(a\leq y)(z) = f(a\leq y)\circ f(x\leq a)(\alpha)$. This implies that $z=f(x\leq a)(\alpha)+\xi$ for $\xi\in \Ker[a,d]$, as desired.

As the above cases are now all settled, we are done.
\qed \end{proof}

Having now established some basic properties of the saecular lattice, we introduce a second lattice.

\begin{definition}
\label{def:downsets}
Let $n_1$ and $n_2$ be non-negative integers.
Define $\mathbf{n_1n_2}:=\mathbf{n_1}\times\mathbf{n_2}$ as the product poset of $\mathbf{n_1}$ and $\mathbf{n_2}$. That is, given $(p,q), (p',q')\in \mathbf{n_1n_2}$, we have that $p'\leq p$ and $ q'\leq q$ if and only if $(p',q') \le (p,q)$. A subset of $S\subseteq \mathbf{n_1n_2}$ is \emph{down-closed}  if $(p,q)\in S$ implies that $(p',q') \in S$ for all $(p',q') \le (p,q)$. We write $\mathrm{DownSets}(\mathbf{n_1n_2})$ for the poset of \emph{down-closed subsets} of $\mathbf{n_1n_2}$, ordered under inclusion. The poset $\mathrm{DownSets}(\mathbf{n_1n_2})$ forms a lattice, where the meet and join operations are intersection and union, respectively. 
\end{definition}

The $\mathrm{DownSet}$ and saecular lattices can be related with the following theorem, which is a consequence of a result in Birkhoff \cite{birkhoff_lattice_1995}.

\begin{theorem}
For each $a \in \mathbf{m}$, there exists a (unique) lattice homomorphism $\Lambda_a:\mathrm{DownSet}(\mathbf{mm})\to \mathrm{ImKer}(f_a)$ such that
\begin{align*}
    \mathbf{xm}\mapsto \im[x,a]
    &&
    \mathbf{my}\mapsto \Ker[a,y]
\end{align*}. 
\end{theorem}
\begin{proof}

	This proof uses the modularity\footnote{A lattice $\mathcal{L}$ is \emph{modular} if, for any $\ell_1,\ell_2,\ell_3\in \mathcal{L}$ such that  $\ell_1 \le \ell_3$, we have  $\ell_1 \vee (\ell_2 \wedge  \ell_3) = (\ell_1 \vee \ell_2) \wedge  \ell_3$.} of submodule lattices (see \cite[Theorem 2.2.6(a)]{grandis_homological_2012}). 
	
	Consider the lattice $\mathrm{Sub}(f_a)$ of submodules of $f_a$.
	Let $p$ and $q$ be the chains $p_0\leq \ldots \leq  p_m$ and $q_0 \leq \ldots \leq q_m$, respectively. Define lattice homomorphisms $\lambda_p:p\to \mathrm{Sub}(f_a)$ and $\lambda_q:q\to \mathrm{Sub}(f_a)$ by $\lambda_p(p_x)=\im[x,a]$ and $\lambda_q(q_y)=\Ker[a,y]$, respectively. 
	By \cite[Section III.7, Theorem 9]{birkhoff_lattice_1995}, there exists a unique lattice homomorphism $\Lambda_a:\textrm{DownSet}(\mathbf{mm}) \to \mathrm{Sub}(f_a)$ such that the following diagram commutes:
	\[\begin{tikzcd}
	&& {\mathrm{DownSet}(\mathbf{m}\mathbf{m})} \\
	p && \mathrm{Sub}(f_a) && q
	\arrow["\psi_p",from=2-1, to=1-3]
	\arrow["\psi_q"',from=2-5, to=1-3]
	\arrow["\Lambda_a"', dashed, from=1-3, to=2-3]
	\arrow["{\lambda_p}", from=2-1, to=2-3]
	\arrow["{\lambda_q}"', from=2-5, to=2-3]
	\end{tikzcd}\] 
	In this diagram, $\psi_p:p\to \textrm{DownSet}(\mathbf{mm})$ and $\psi_q:q\to \textrm{DownSet}(\mathbf{mm})$ are defined by $\psi_p(p_x)=\mathbf{xm}$ and $\psi_q(q_y)=\mathbf{my}$, respectively.
	Moreover, the codomain of $\Lambda_a$ can be restricted to $\ImKer{f_a}$ because the images of $\lambda_p$ and $\lambda_q$ are both in $\ImKer{f_a}$. We therefore have a unique lattice homomorphism $\Lambda_a:\mathrm{DownSet}(\mathbf{mm})\to \mathrm{ImKer}(f_a)$, as desired.
\qed \end{proof}

\begin{definition}
We call $\Lambda_a:\mathrm{DownSet}(\mathbf{mm})\to\mathrm{ImKer}(f_a)$  the \emph{saecular homomorphism} correstricted to $a$. 
\end{definition}

The saecular homomorphisms allow us to define a pair of commutative diagrams, which we introduce in Theorems \ref{thm:comm_diagram_forward} and \ref{thm:comm_diagram_backward}.  These diagrams have great utility; for example,  we will use them to show that the direct image operator $H \mapsto f(a \le b)(H)$ and the inverse image operator $H \mapsto f(a \le b)^{-1}(H)$ induce  homomorphisms between saecular lattices.

\begin{theorem}
\label{thm:comm_diagram_forward}
Fix $a,b\in\mathbf{m}$ such that $a\leq b$.  The following diagram commutes:
\[
\begin{tikzcd}
	&& S & {S\cap \mathbf{am}} \\
	{\mathbf{xm}} & {\mathbf{my}} & {\mathrm{DownSets}(\mathbf{mm})} & {\mathrm{DownSets}(\mathbf{mm})} & {\mathbf{xm}} & {\mathbf{my}} \\
	{\mathrm{Im}[x,a]} & {\Ker[a,y]} & {\mathrm{ImKer}(f_a)} & {\mathrm{ImKer}(f_b)} & {\mathrm{Im}[x,b]} & {\Ker[b,y]} \\
	&& H & {f(a\leq b)(H)}
	\arrow["g", from=2-3, to=2-4]
	\arrow["\Lambda_b", from=2-4, to=3-4]
	\arrow["\Lambda_a", from=2-3, to=3-3]
	\arrow["h", from=3-3, to=3-4]
	\arrow[maps to, from=4-3, to=4-4]
	\arrow[maps to, from=2-2, to=3-2]
	\arrow[maps to, from=1-3, to=1-4]
	\arrow[maps to, from=2-1, to=3-1]
	\arrow[maps to, from=2-5, to=3-5]
	\arrow[maps to, from=2-6, to=3-6]
\end{tikzcd}
\]
\end{theorem}

\begin{proof}
We first consider sets of the form $\mathbf{cm}\cap \mathbf{md}$. On one hand, Lemma \ref{lemma:second_annoying_lemma} implies that $$
\mathbf{cm}\cap \mathbf{md}\xmapsto{\quad\Lambda_a\quad} \mathrm{Im}[c,a]\cap \Ker[a,d]\xmapsto{\quad h \quad } \mathrm{Im}[c,b]\cap\Ker[b,d]\cap \mathrm{Im}[a,b]\,.
$$
On the other hand,
$$
\mathbf{cm}\cap \mathbf{md}\xmapsto{\quad g \quad} \mathbf{cm}\cap \mathbf{md}\cap\mathbf{am}\xmapsto{ \quad \Lambda_b \quad } \mathrm{Im}[c,a]\cap\Ker[b,d]\cap \mathrm{Im}[a,b]\,.
$$
Therefore, $h\circ\Lambda_a(\mathbf{cm}\cap \mathbf{md})=\Lambda_b\circ g(\mathbf{cm}\cap \mathbf{md})$. 
All sets in $\mathrm{DownSets}[\mathbf{mm}]$ can be formed as unions of sets of the form $\mathbf{cm}\cap \mathbf{md}$. 
Because $g$, $\Lambda_a$, and $\Lambda_b$ are lattice homomorphisms, and $h$ preserves join (because the image of unions of sets is equal to the union of images of sets), it follows that $h\circ \Lambda_a$ and $\Lambda_b\circ g$ agree on every set in $\mathrm{Downsets}(\mathbf{mm})$, so the diagram commutes.
\qed \end{proof}

We have a similar commutative diagram for $f(a\leq b)^{-1}$.
\begin{theorem}
\label{thm:comm_diagram_backward}
Fix $a,b\in\mathbf{m}$ such that $a\leq b$. The following diagram commutes: 
\[\begin{tikzcd}
	&& {S\cup\mathbf{ma}} & S \\
	{\mathbf{xm}} & {\mathbf{my}} & {\mathrm{DownSets}(\mathbf{mm})} & {\mathrm{DownSets}(\mathbf{mm})} & {\mathbf{xm}} & {\mathbf{my}} \\
	{\mathrm{Im}[x,a]} & {\Ker[a,y]} & {\mathrm{ImKer}(f_a)} & {\mathrm{ImKer}(f_b)} & {\mathrm{Im}[x,b]} & {\Ker[b,y]} \\
	&& {f(a\leq b)^{-1}(H)} & H
	\arrow["g"', from=2-4, to=2-3]
	\arrow["\Lambda_b", from=2-4, to=3-4]
	\arrow["\Lambda_a", from=2-3, to=3-3]
	\arrow["h"', from=3-4, to=3-3]
	\arrow[maps to, from=4-4, to=4-3]
	\arrow[maps to, from=2-2, to=3-2]
	\arrow[from=1-4, to=1-3]
	\arrow[maps to, from=2-1, to=3-1]
	\arrow[maps to, from=2-5, to=3-5]
	\arrow[maps to, from=2-6, to=3-6]
\end{tikzcd}\]
\end{theorem}

\begin{proof}
First, we consider sets of the form $\mathbf{cm}\cup\mathbf{md}$.  On one hand, 
$$
\mathbf{cm}\cup\mathbf{md}
\xmapsto{\quad g \quad } 
\mathbf{cm}\cup\mathbf{md} \cup \mathbf{ma}
\xmapsto{\quad \Lambda_a \quad }
\mathrm{Im}[c,a]+\Ker[a,d]+\Ker[a,b]\,.
$$
On the other hand, by Lemma \ref{lemma:annoying_lemma_backwards}, applying $\Lambda_b$ followed by $h$ yields 
$$
\mathbf{cm}\cup\mathbf{md}
\xmapsto{\quad \Lambda_b\quad} 
\mathrm{Im}[c,b]+\Ker[b,d]
\xmapsto{\quad h \quad }
\mathrm{Im}[c,a]+\Ker[a,d]+\Ker[a,b]. 
$$
Thus, for sets of the form $\mathbf{cm}\cup\mathbf{md}$, we see that $\Lambda_a\circ g=h\circ\Lambda_b$. 

Any set in $\mathrm{DownSets}(\mathbf{mm})$ can be written as an intersection of sets of the form $\mathbf{cm}\cup\mathbf{md}$. Because $g$, $\Lambda_a$, and $\Lambda_b$ are lattice homomorphism and $h$ preserves meet (because the preimage of an intersection is equal to the intersection of preimages), we see that $\Lambda_a\circ g = h\circ \Lambda_b$ for every set in $\mathrm{DownSets}(\mathbf{mm})$.
\qed \end{proof}

\begin{rmk}
Theorems \ref{thm:comm_diagram_forward} and \ref{thm:comm_diagram_backward} imply that the function $h$ (which is given by applying $f(a\leq b)$ and $f(a\leq b)^{-1}$, respectively) in both diagrams are lattice homomorphisms. This is because, in each diagram, the functions $g$, $\Lambda_a$, and $\Lambda_b$ lattice homomorphisms and $\Lambda_a$ and $\Lambda_b$ are surjective.  
\end{rmk}

\subsection{Complements in the saecular submodule lattices}
\label{sec:complements_subgroup_lattice}

In this section, we show that if the cokernel of $f(a \le b)$ is free for all $a$ and $b$ such that $0\le a \le b\leq m$, then for any $f_c$ and $A,B\in \ImKer(f_c)$ such that $A\subseteq B$, the cokernel of the inclusion $A \subseteq B$ (i.e., the quotient $B/A$) is free.
This implies that whenever we have an inclusion $A\subseteq B$ between elements of a saecular lattice, there exists a complementary submodule $C$  such that $A \oplus C = B$. 
This result builds upon and generalizes the complement results in Section \ref{sec:complements}.

\begin{definition}
    We say that $C$ \emph{complements} $A \su B$ (or that $C$ is a \emph{complement} of $A$ in $B$) whenever $A \oplus C = B$.
\end{definition}

In Section \ref{sec:existence_interval_decomp}, we will use these complements to explicitly construct interval decompositions.

As in previous sections, we assume that the cokernel of every structure map $f(a \le b)$ is free.  Thus,
$$
f_b\cong \mathrm{Im}[a,b]\oplus \Coker[a,b]\,.
$$

Using Theorem \ref{thm:complement_special}, as well as some lemmas that we state and prove shortly, we have the following general result.

\begin{theorem}\label{thm:complement_general}
Let $a\in\mathbf{m}$ and $L,H\in\ImKer(f_a)$ such that $L\subseteq H$. Then $L$ has a complement in $H$. 
\end{theorem}

We prove Theorem \ref{thm:complement_general} using the results of Lemmas
\ref{lemma:complement_equivalence}, \ref{lemma:intersection_closure},
\ref{lemma:complement_shrink}, and \ref{lemma:interection_generator}.

\begin{lemma}\label{lemma:interection_generator}
Every element of the saecular lattice $\ImKer(f_a)$ can be expressed as an intersection of submodules of the form $\mathrm{Im}[c,a]+\Ker[a,d]$ for some choice of $c,d\in\mathbf{m}\cup\{0\}$. 
\end{lemma}

\begin{proof}
Submodules of the form $\mathrm{Im}[c,a]+\Ker[a,d]$ are precisely the image of subsets of the form $\mathbf{cm}\cup\mathbf{md}$ under $\Lambda_a$. 
Because any set in $\mathrm{DownSets}(\mathbf{mm})$ can be written as an intersection of sets of the form $\mathbf{cm}\cup\mathbf{md}$ (noting that $c$ or $d$ can be $0$) and $\Lambda_a$ is a lattice homomorphism, it follows that any submodule of $f_a$ can be written as an intersection of submodules of the form $\mathrm{Im}[c,a]+\Ker[a,m]$. 
\qed \end{proof}

 Lemmas
\ref{lemma:complement_equivalence}, \ref{lemma:intersection_closure},
\ref{lemma:complement_shrink}, and \ref{lemma:interection_generator}, along with Theorem \ref{thm:complement_special}, allow us to prove Theorem \ref{thm:complement_general}.

\begin{proof}[Theorem \ref{thm:complement_general}.]
By Lemma \ref{lemma:interection_generator}, every $L\in\ImKer(f_a)$ can be represented as an intersection of submodules of the form $\mathrm{Im}[c,a]+\Ker[a,d]$. By Theorem \ref{thm:complement_special} and Lemma \ref{lemma:intersection_closure}, $L$ admits a complement in $f_a$. By Lemma \ref{lemma:complement_shrink}, $L$ has a complement in $H$, as desired. 
\qed \end{proof}

\subsection{Existence of interval decompositions}
\label{sec:existence_interval_decomp}

We now use the saecular lattice (see Definition \ref{def:saecular_lattice}) of each $f_a$ and the existence of complements in the saecular lattice (see Theorem \ref{thm:complement_general}) to prove the existence of interval decompositions. 
We once again assume that all structure maps $f(a\leq b)$ have free cokernel. 

Our argument proceeds as follows.  First, we use the results of Subsection \ref{sec:complements_subgroup_lattice} to obtain a direct-sum decomposition $f_a \cong \bigoplus_{i,j}A^{ij}_a$ for all $a$, where each $A^{ij}_a$ is defined as a complement in the saecular lattice (see Theorem \ref{thm:decomposition_ij}).  
We then modify the summands in each decomposition so that 
$f(a \le b)(A^{ij}_a) \in \{ A^{ij}_b, 0 \}$ whenever $a \le b$ (see Lemma \ref{lemma:in_bounds_isomorphism}).  
The resulting family of summands $A^{ij}_a$ then naturally combines to form an interval decomposition of $f$.

\begin{lemma}\label{lemma:singleton_complement}
Suppose that we have a diagram of inclusions of down-closed sets in $\mathbf{mm}$
\[\begin{tikzcd}
	Y & T \\
	X & S\,,
	\arrow[hook, from=2-1, to=1-1]
	\arrow[hook, from=2-1, to=2-2]
	\arrow[hook, from=2-2, to=1-2]
	\arrow[hook, from=1-1, to=1-2]
\end{tikzcd}\]
where $Y = X \sqcup \{(i,j)\}$ and $T = S \sqcup \{(i,j)\}$.  Let $A^{ij}_a$ be a complement of $\Lambda_a(X)$ in $\Lambda_a(Y)$.  Then $A^{ij}_a$ is also a complement of $\Lambda_a(S)$ in $\Lambda_a(T)$. 
\end{lemma}

\begin{proof}
Note that $A^{ij}_a$ is a submodule of $\Lambda_a(T)$. Moreover,   $Y\cap S = X$ and $Y \cup S = T$. From this, we have
\begin{align}
    \frac{\Lambda_a(Y)}{\Lambda_a(X)} &=\frac{\Lambda_a(Y)}{\Lambda_a(Y\cap S)} \notag \\
    &=\frac{\Lambda_a(Y)}{\Lambda_a(Y)\cap \Lambda_a(S)} \notag \\
    &\cong \frac{\Lambda_a(Y)+\Lambda_a(S)}{\Lambda_a(S)} \label{eq_needs_number}\\
    &= \frac{\Lambda_a(Y\cup S)}{\Lambda_a(S)} \notag\\
    &= \frac{\Lambda_a(T)}{\Lambda_a(S)}\,. \notag 
\end{align}
The isomorphism in \eqref{eq_needs_number} is from the second isomorphism theorem (See \cite[Section 10.2, Theorem 4.2]{dummit_abstract_2004}). We therefore have the commutative diagram 
\[\begin{tikzcd}
	{} & {\Lambda_a(Y)/\Lambda_a(X)} \\
	{A^{ij}_a} & {\Lambda_a(T)/\Lambda_a(S)\,,}
	\arrow[hook, from=2-1, to=1-2]
	\arrow[from=2-1, to=2-2]
	\arrow["\cong"', from=1-2, to=2-2]
\end{tikzcd}\]
so the map $A^{ij}_a\to \Lambda_a(T)/\Lambda_a(S)$ is injective. Therefore, $\Lambda_a(S)\cap A^{ij}_a=0$. 
Moreover, $\Lambda_a(S)+ A^{ij}_a=\Lambda_a(T)$, because
\begin{align*}
    \Lambda_a(S)+A^{ij}_a&=\Lambda_a(X)+ A^{ij}_a+\Lambda_a(S) \\
    &=\Lambda_a(Y)+\Lambda_a(S)\\
    &=\Lambda_a(Y \cup S) \\
    &=\Lambda_a(T)\,.
\end{align*}
As a result, $A^{ij}_a$ is a complement of $\Lambda_a(S)$ in $\Lambda_a(T)$, as desired. 
\qed \end{proof}

For each $f_a$, the submodule $A^{ij}_a$ in Lemma \ref{lemma:singleton_complement} can identified with $(i,j)\in \mathbf{mm}$. 
In fact, the set of $A^{ij}$ are summands of $f_a$. One can also use these summands of $f_a$ to construct any submodule in $\ImKer(f_a)$.

\begin{framed}
Lemma \ref{lemma:singleton_complement} tells us that  every complement of $\Lambda_a(X) \su \Lambda_a(Y)$ is also a complement of $\Lambda(S) \su \Lambda(T)$ when $X$ and $Y$ are small enough with respect to inclusion (i.e., satisfy the relation in Lemma \ref{lemma:singleton_complement}).

For the remainder of this section, we fix a family of complements, with $X$ and $Y$ as small as possible.  
That is, for each $a,i,j \in \mathbf{m}$, choose a direct sum decomposition 
\begin{align}
\Lambda_a(Y) = \Lambda_a(X) \oplus A^{ij}_a
\label{eq:aij_constraints}
\end{align}
where $Y = \mathbf{ij} = \mathbf{im}\cap\mathbf{mj}$ and $X = Y \setminus \{(i,j)\}$.
\end{framed}

Our next technical result is Lemma \ref{lem:emptycap}. Viewed at a high level, this result states that when a space of the form $\mathrm{Im}[x,a]$ and a space of the form $\mathrm{Ker}[a,y]$ are each large enough to contain $A^{ij}_a$, then so is their intersection. However, decrementing $x$ and $y$ by a single unit can make the image and kernel spaces small enough so that they do not include any nonzero element of $A^{ij}_a$, even when we take the sum of the two spaces. 

\begin{lemma}
\label{lem:emptycap}
For each $(i,j) \in \mathbf{mm}$, we have
\begin{align}
	0 & = A^{ij}_a \cap \big ( \mathrm{Im}[i-1,a] + \Ker[a,j-1] \big )\,, \label{eq:emptycap1} \\
	A^{ij}_a & =  A^{ij}_a \cap \big(   \mathrm{Im}[i,a] \cap \Ker[a,j] \big ) \hspace{1cm}\textrm{(equivalently, $A^{ij}_a \su    \mathrm{Im}[i,a] \cap \Ker[a,j]\,.$)} \label{eq:emptycap2}
\end{align}
\end{lemma}
\begin{proof}
Define down-closed sets $X\su Y$ and $S \su T$ such that
\begin{align*}
X & = Y \setminus \{(i,j)\} = Y \cap S\,, \\
Y & = \mathbf{ij}\,, \\
S  & = \{(p,q) \in \mathbf{mm} : p < i \text{ or } q < j\} =  \mathbf{(i-1)}\mathbf{m}\cup\mathbf{m}\mathbf{(j-1)}\,, \\
T &  = S \cup \{(i,j)\} = S \cup Y\,.
\end{align*}
Then $A^{ij}_a$ complements $\Lambda_a(S) \su \Lambda_a(T)$, where $\Lambda_a(S) = \mathrm{Im}[i-1,a] + \Ker[a,j-1]$, so Equation
\eqref{eq:emptycap1} follows. 
Likewise, $A^{ij}_a$ complements $\Lambda_a(X) \su \Lambda_a(Y)$, where $\Lambda_a(Y) = \mathrm{Im}[i,a] \cap  \Ker[a,j]$. 
Therefore, $A^{ij}_a\subseteq \mathrm{Im}[i,a] \cap  \Ker[a,j]$, so Equation \eqref{eq:emptycap2} follows.
\qed \end{proof}

\begin{theorem}\label{thm:decomposition_ij}
For every down-closed subset $S \su \mathbf{mm}$,  we have
\[
\Lambda_a(S)=\bigoplus_{(i,j)\in S}A^{ij}_a \, .
\] 
\end{theorem}

\begin{proof}
Choose a sequence of down-closed sets $\emptyset=S_0\subseteq S_1\subseteq \cdots\subseteq S_n=S$ such that each set difference $S_p \setminus S_{p-1}$ is a singleton $\{(i_p,j_p)\}$. The submodule $A^{i_p,j_p}_a$ complements $\Lambda_a(S_{p-1})\subseteq\Lambda_a(S_p)$ for all $p$, by Lemma \ref{lemma:singleton_complement}. Therefore, $$\Lambda_a(S)= \bigoplus_pA^{i_p,j_p}_a=\bigoplus_{(i,j)\in S} A^{ij}_a.$$ 
\qed \end{proof}

In the special case where $f$ is a persistent homology module,  one can view each $A_a^{ij}$ as homology classes that are born at $i$ and die at $j$. 
This interpretation is valid by Lemma \ref{lem:emptycap}, because $A_a^{ij}\subseteq \Ker[a,j]$, $A_a^{ij}\cap\Ker[a,j-1]=0$, $A_a^{ij}\subseteq \im[i,a]$, and $\im[i-1,a]\cap A_a^{ij}=0$. 
Lemmas \ref{lemma:out_of_bounds} 
and \ref{lemma:in_bounds_isomorphism} provide additional context for this interpretation. They also provide the necessary ingredients to prove Theorem \ref{thm:hard_direction_result}.
In addition, in the case where $f$ represents the persistent homology module of a filtered simplicial complex, the submodules $A_a^{ij}$ are canonically isomorphic\footnote{Note, however, that they are not precisely the same; in particular, the pair groups always exist, while existence of the summands $A_a^{ij}$ requires cokernels of structure maps to be free. Moreover, the pair groups are quotients of subgroups of $f_a$, while the summands $A_a^{ij}$ are simply subgroups.} to the ``pair groups'' discussed in \cite{cohen-steiner_extending_2009}. 

\begin{lemma}\label{lemma:out_of_bounds}
If $a \notin [i,j)$, then $A_a^{ij}=0$.
\end{lemma}

\begin{proof}
On one hand, if $a<i$, then $\im[i,a]=f_a = \im[i-1,a]$. 
Because we also have $\im[i-1,a]\cap A_a^{ij}=0$, it follows that $A_a^{ij}=0$. 

On the other hand, $a\geq j$ implies that $\Ker[a,j]=0$. 
Because $A_a^{ij}\subseteq\Ker[a,j]$, it follows that $A_a^{ij}=0$. 
\qed \end{proof}

\begin{lemma}\label{lemma:in_bounds_isomorphism}
Suppose that $a,b \in [i,j)$, where $a\leq b$. Define $Y = \mathbf{ij}$ and $X = Y \setminus \{(i,j)\}$.
Then
$$
\Lambda_b(Y) = \Lambda_b(X) \oplus f(a\le b)(A^{ij}_a)\,.
$$
That is, the direct image $f(a\le b)(A^{ij}_a)$ satisfies the same condition (namely, Equation \eqref{eq:aij_constraints}) that we impose on $A^{ij}_b$ with $b$ in place of $a$.
Furthermore, $f(a \le b)$ restricts to an isomorphism $A^{ij}_a \to f(a\le b)(A^{ij}_a)$.
\end{lemma}

\begin{proof}
    Define 
    \begin{align*}
        X' &= \Lambda_a(X)\,,
        &
        Y' &= \Lambda_a(Y)\,,       
        \\
        X'' &= \Lambda_b(X)\,,        
        &
        Y'' &= \Lambda_b(Y) \, .                
    \end{align*}
Let $g$ be the map $S \mapsto S \cap \mathbf{am}$, and let $h$ be the direct image operator $H \mapsto f(a \le b)(H)$.    
We have $Y' = X' \oplus A^{ij}_a$ by hypothesis.
First, we wish to show that $Y'' = X'' \oplus h(A^{ij}_a)$.
\\    
Because $g(X) = X$ and $g(Y) = Y$, it follows from Theorem \ref{thm:comm_diagram_forward} that $h(X') = (h \circ \Lambda_a)(X) = (\Lambda_b \circ g)(X) =  X''$ and $h(Y')  = (h \circ \Lambda_a)(Y) = (\Lambda_b \circ g)(Y) = Y''$.  
Additionally, because the kernel of $f(a \le b)|_{Y'}$ equals $\Ker[a,b] \cap Y' = \Lambda_a(\mathbf{mb}) \cap \Lambda_a(Y) = \Lambda_a(\mathbf{mb} \cap Y)$, which is a submodule of $\Lambda_a(X) = X'$, we have 
$$
\textstyle
Y'' 
= 
f(a\le b)( Y') = f(a\le b)( X' \oplus A^{ij}_a) = h(X') \oplus h(A^{ij}_a)
=
X'' \oplus h(A^{ij})\,,
$$
as desired.  In addition, because the kernel of $f(a \le b)|_{Y'}$ is a submodule of $X'$, it follows that $f(a \le b)$ restricts to an isomorphism $A^{ij}_a \to h(A^{ij}_a)$.  The desired conclusion follows.
\qed \end{proof}

To prove Theorem \ref{thm:hard_direction_result}, we first introduce the notion of generalized interval modules, which we will show decompose as a direct sum of interval modules. 

\begin{definition}
    Let $R$ be a commutative ring. A \emph{generalized interval module} over $R$ is a persistence module $J^{ij}:\{0,1,\ldots,m\}\to R\textrm{-Mod}$ such that  $J^{ij}_a = 0$ for $a \notin [i,j)$ and $J^{ij}(a \le b):J^{ij}_a\to J^{ij}_b$ is an isomorphism whenever $a,b \in [i,j)$. 
\end{definition}

\begin{proposition}
\label{prop:generalized_interval_module}
    Let $J^{ij}$ be a generalized interval module. Then $J^{ij}$ decomposes as a direct sum of interval modules.  
\end{proposition}

\begin{proof}
    Choose any basis $\beta_i$ of $J^{ij}_i$. For $a\notin[i,j)$, let $\beta_a = \emptyset$; for $a\in (i,j)$, let $\beta_a = J^{ij}(i\leq a)(\beta_i)$.  
	
	It suffices to show that for all $a$, we have that $\beta_a$ is a basis of $J_a^{ij}$ and that $J^{ij}(a\leq b)$ maps $\beta_a\setminus\ker(J^{ij}(a\leq b))$ injectively into $\beta_b$ for all $a \le b$.
	
	When $a$ or $b$ are outside $[i,j)$, this is trivial, as one of $\beta_a$ or $\beta_b$ must be empty. When $i\leq a\leq b<j$, we note that $J^{ij}(a\leq b)$ is an isomorphism.
	Therefore, each $\beta_a$ for $a\in[i,j)$ is a basis and $J^{ij}(a\leq b)$ is a bijection between $\beta_a$ and $\beta_b$, which satisfies our desired claim. 
\qed \end{proof}

We conclude this section by proving Theorem \ref{thm:hard_direction_result}, thus completing the proof of Theorem \ref{thm:problem_statement}.

\begin{proof}[Theorem \ref{thm:hard_direction_result}]
    It suffices to show that $f$ decomposes as a direct sum of generalized interval modules, which decompose into interval modules by Proposition \ref{prop:generalized_interval_module}.
	By Lemmas \ref{lemma:out_of_bounds} and \ref{lemma:in_bounds_isomorphism}, we can assume without loss of generality that
	\begin{align*}
	A^{ij}_a
	=
	\begin{cases}
	f(i \le a)(A^{ij}_i)\,, & a\in [i,j)\,, \\
	0\,, & \text{ otherwise}\,.
	\end{cases}
	\end{align*}
	Additionally, Lemma \ref{lemma:in_bounds_isomorphism} implies that $f(a \le b)|_{A^{ij}_a}: A^{ij}_a \to A^{ij}_b$ is an isomorphism whenever $a, b \in [i,j)$ and is the zero map otherwise.     
	It therefore follows that, for each $i$ and $j$, we have a persistence submodule $A^{ij}$ such that $(A^{ij})_{a}=A^{ij}_a$ and $A^{ij}(a\leq b) = f(a\leq b)|_{A^{ij}_a}$. 
	Thus, $A^{ij}$ is a generalized interval module.
	For each $a$, we have 
	\begin{align*}
	f_a &= \bigoplus_{ij}A^{ij}_a \hspace{1cm} \textrm{(by Theorem \ref{thm:decomposition_ij})}\\ 
	&=\bigoplus_{ij}(A^{ij})_a\,. 
	\end{align*}
	Additionally, for each $a$ and $b$ such that $a\leq b$, we have
	\begin{align*}
	f(a\leq b) &=  \bigoplus_{i,j}f(a\leq b)|_{A^{ij}_a}\\
	&= \bigoplus_{i,j}A^{ij}(a\leq b)\,.
	\end{align*}
	Therefore, $f$ is a direct sum of the generalized interval modules $A^{ij}$, as desired.
\qed \end{proof}

\section{Smith Normal Form}
\label{appendix:smith_normal_form}

The decomposition procedure in Algorithm \ref{alg:matrix} requires no special machinery except the Smith normal form.  In this section, we provide a cursory review of the relevant facts about the Smith normal form. See \cite{dummit_abstract_2004} for a detailed introduction

We say that a matrix $S$ with entries in a PID $R$ is \emph{unimodular} if there exists a matrix $S^{-1}$ with entries in $R$ such that $SS^{-1} = S^{-1}S$ is the identity.

A Smith-normal-form factorization of a matrix $A \in M_{m,n}(R)$ is an equation
$$
SAT = D\,,
$$
where $S$ and $T$ are unimodular and $D$ is a diagonal matrix of form $\mathrm{diag}(\alpha_1, \ldots, \alpha_r, 0 , \ldots 0)$ in which $\alpha_i$ divides $\alpha_{i+1}$ for all $i < r$. 
We refer to $\alpha_1,\ldots,\alpha_r$ as the \emph{elementary divisors} of $A$. 
We say that $A$ has \emph{unit elementary divisors} if its elementary divisors are all units. 

We refer to the column submatrix of $T$ that consists of the first $r$ columns as the \emph{positive} part of $T$; we denote it by $T^+$.  The submatrix that consists of the remaining $m-r$ columns is the \emph{negative} part of $T$; we denote it by $T^-$. These matrices are complementary in the sense of Definition \ref{def:matrixcomplements}.

\begin{definition}
\label{def:matrixcomplements}
    Let $A\in M_{p,q}(R)$, where $p\geq q$ and $A$ has unit elementary divisors. We say that $B\in M_{p,p-q}(R)$ is a \emph{complement} of $A$ if $[A|B]$ is unimodular.
\end{definition}

\begin{rmk}
    Let $A\in M_{m,n}(R)$, and let $SAT=D$ be a Smith-normal-form factorization. Then $AT^+$ has linearly-independent columns and $AT^-=0$. 
\end{rmk}

We similarly define the positive part of $S^{-1}$ to be the submatrix $(S^{-1})^+$ composed of the first $r$ columns and the negative part of $S^{-1}$ to be the submatrix $(S^{-1})^-$ composed of the last $m-r$ columns. 
These submatrices bear a special  relationship to the cokernel of $A$, as formalized in Proposition \ref{prop:snf_and_cokernels}.  
The proof consists of straightforward calculations in matrix algebra.

\begin{proposition}
\label{prop:snf_and_cokernels}
    Fix a matrix $A\in M_{m,n}(R)$, and let $SAT=D$ be a Smith-normal-form factorization of $A$. Then the following are equivalent.
    \begin{enumerate}
        \item The homomorphism represented by $A$ has a free cokernel.
        \item The diagonal entries of $D$ are units or $0$.
    \end{enumerate}
When these conditions hold, 
\begin{enumerate}
    \item The image of $A$ is the column space of  $(S^{-1})^+$. 
    \item The column space of $(S^{-1})^-$ is a complement of the image space of $A$.
\end{enumerate}
 
\end{proposition}

The Smith normal form also has a useful relationship with nested kernels.  
Given a persistence module $f:\{0,\ldots,m\}\to R\textrm{-Mod}$, recall that $F_a$ is the matrix representation of the map $f(a \le a+1).$

\begin{proposition}
\label{prop:kernel_filtration_matrix}
	Given $a\in\{0,\ldots,m\}$, let $S_aF_aT_a$ be a Smith-normal-form factorization of $F_a$, and let $S_k$ and $T_k$ be inductively defined such that
	\begin{align}
	S_k \Big(F_k \cdots F_a T_a^+ \cdots T_{k-1}^+ \Big) T_k
	\label{eq:smothnormalrepeated}
	\end{align}
	is a Smith-normal-form factorization of the product matrix $F_k \cdots F_a T_a^+ \cdots T_{k-1}^+$ for all $k > a$.  If $X^k := T_a^+ \cdots T_{k-1}^+ T_k^-$, then the columns of $[X^{a+1} | \cdots | X^{k}]$ form a basis of the kernel of $f(a \le k)$ for all $k > a$.
\end{proposition}
\begin{proof}
    For any $k$, we have that $F_k \cdots F_a \cdot [X^{k} | \cdots | X^{a+1}] = 0$. This is because, by construction, $F_r\cdots F_aX_r = F_r \cdots F_a T_a^+ \cdots T_{r-1}^+ T_r^-=0$ for all $r$ satisfying $a < r\leq k$. This implies that the columns of $[X^k|\cdots|X^{a+1}]$ are in the kernel of $F_k\cdots F_a$. 

    \begin{figure}
    \centering
    \includegraphics[width=0.8\columnwidth,]{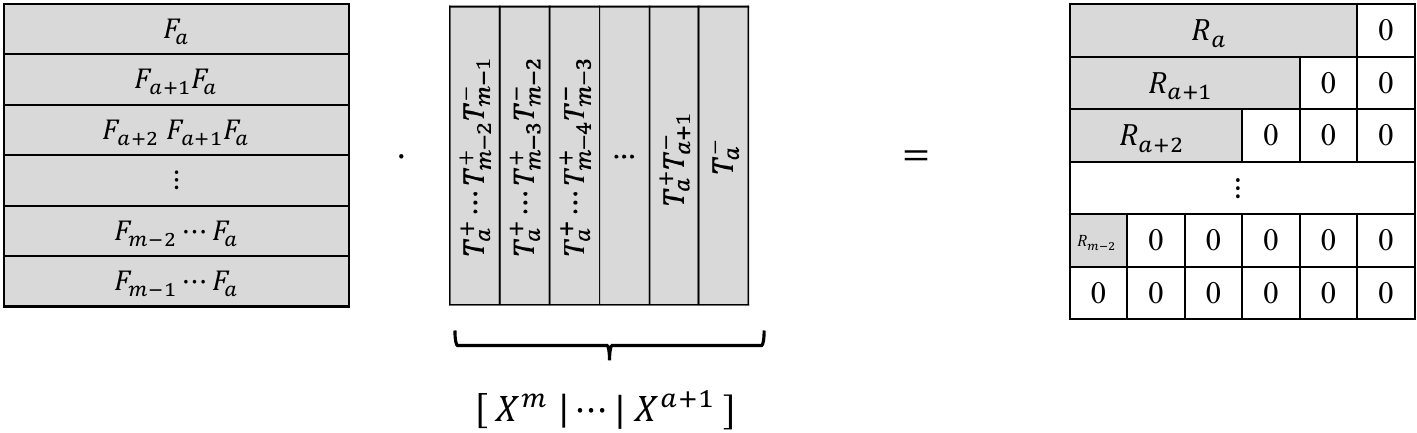}
    \caption{A product of matrices with block structure. The matrix $R_k$ has linearly independent columns and can be expressed in the form \eqref{eq:defineRk}. Observe that $R_{m-1}=0$ because $F_{m-1}=0$.}
    \label{fig:filtered_kernel_basis}
\end{figure}
    
    A straightforward calculation yields\footnote{Observe that $T_{m-1}^-=T_{m-1}$ because $F_{m-1}$ is the zero map.}
    \begin{align}
    [X^{m} | \cdots | X^{a+1}] = T_a \cdot \diag(T_{a+1}, \id) \cdots \diag(T_m, \id)\,,
    \label{eq:productofdiagonals}
    \end{align}
    where $\diag(T_{k}, \id)$ denotes the block-diagonal matrix with $T_k$ in the upper-left corner and an identity matrix of appropriate size in the lower-right corner to form a matrix of the same size as $T_a$.
    Therefore, $[X^{m} | \cdots | X^{a+1}]$ is unimodular because it is a product of unimodular matrices. 
    Consequently, for any $k$, the matrix $[X^{k} | \cdots | X^{a+1}]$ has linearly independent columns.

    It remains to show that the columns of $[X^{k} | \cdots | X^{a+1}]$ span the kernel of $F_k\cdots F_a$.
    Let 
    \[
    \tilde{F}=\begin{bmatrix}
        F_a \\ F_{a+1}F_a \\ \vdots \\ F_{m-2}\cdots F_a \\ F_{m-1}\cdots F_a
    \end{bmatrix}
    \]
    and consider the product $\tilde{F}[X^{m} | \cdots | X^{a+1}]$, which is of the form in Figure \ref{fig:filtered_kernel_basis}.  
    It remains to show that the columns of $[X^{k} | \cdots | X^{a+1}]$ span the kernel of $F_k\cdots F_a$.
	Let 
	\[
	\tilde{F}=\begin{bmatrix}
	F_a \\ F_{a+1}F_a \\ \vdots \\ F_{m-2}\cdots F_a \\ F_{m-1}\cdots F_a
	\end{bmatrix}
	\]
	and consider the product $\tilde{F}[X^{m} | \cdots | X^{a+1}]$, which is of the form in Figure \ref{fig:filtered_kernel_basis}.  
	
	For each $k$, we can write 
	\begin{align}
	R_k=F_k\cdots F_aT_a^+\cdots T_k^+[\; T_{k+1}^+\cdots T_{m-2}^+T_{m-1}^- \;| \;  T_{k+1}^+\cdots T_{m-3}^+ T_{m-2}^-  \;|\;  \cdots  \;|\;  T_{k+1}^- \;]\,.
	\label{eq:defineRk}
	\end{align}
	This implies that the columns of $[X^{k} | \cdots | X^{a+1}]$ span a submodule $L$ of the kernel (which we denote $K$) of $F_k\cdots F_a$.
	Because both matrices $F_k\cdots F_aT_a^+\cdots T_k^+$ and \newline $[\; T_{k+1}^+\cdots T_{m-2}^+T_{m-1}^- \;| \;  T_{k+1}^+\cdots T_{m-3}^+ T_{m-2}^-  \;|\;  \cdots  \;|\;  T_{k+1}^- \;]$ have linearly independent columns, the submatrix $R_k$ also has linearly independent columns.
	Therefore, by counting dimensions, $\dim(L) = \dim(K)$. 
	Therefore, $L = K$ if and only if $K/L$ is torsion-free.  
	Because $[X^{k} | \cdots | X^{a+1}]$ admits a complement --- namely, $[X^m | \cdots | X^{k-1}]$ --- we see that $c \cdot v \in L$ implies that $v \in L$ for every nonzero scalar $c \in R$. 
	Therefore, $K/L$ is torsion-free, so $K=L$. 
	Because the columns of $[X^{k} | \cdots | X^{a+1}]$ are linearly independent, they form a basis of $K$.
    \qed \end{proof}

\section{The standard algorithm: A fast solution for simplex-wise persistent homology}
\label{sec:standardalgorithm}

Obayashi and Yoshiwaki \cite{obayashi_field_2023} provided an algorithm to determine whether $f$ splits as a direct sum of interval modules and, if so, to construct an explicit decomposition. This result was not formally refereed; however their proof appears to be correct and we provide an independent argument in Theorem \ref{thm:StandardAlgorithmGeneralized}.

Here, we provide background information on the algorithm of \cite{obayashi_field_2023}, an extension to general filtered chain complexes, an extension to integer cohomology, and a brief discussion of limitations and failure cases. The algorithm consists of performing a standard algorithm to compute persistent homology for simplex-wise filtrations, with an additional break condition. After a brief comment on time complexity, our first step will be to review the procedure and rephrase the result (Section \ref{sec:standardalgorithmREVIEW}).

\subsection{Time complexity}

The algorithm of \cite{obayashi_field_2023} applies only to persistence modules realized via integer persistent homology of simplicial complexes with simplex-wise filtrations. It does not obviously generalize to arbitrary pointwise free and finitely-generated persistence module over a PID (see Section \ref{sec:obayashi_limitations}). However, in the cases where this result pertains,  the complexity bound of \cite{obayashi_field_2023} is much better than the one afforded by Algorithm \ref{alg:matrix}.
In particular, the algorithm of \cite{obayashi_field_2023} runs in the time it takes to perform Gaussian elimination on the boundary matrix of the filtered simplicial complex (with field coefficients). 
By contrast, to solve the same problem with Algorithm \ref{alg:matrix} one would have to compute the integer homology groups of each space in the filtration as a preprocessing step, then perform a variety of additional matrix operations. A single one of these homology computations may take time comparable to the complete algorithm of \cite{obayashi_field_2023}. 
See Theorems \ref{thm:obayashi_complexity} and \ref{thm:comparative_freeness_complexity} for specifics.

\subsection{The standard algorithm}
\label{sec:standardalgorithmREVIEW}

The standard algorithm (Algorithm \ref{alg:standard}) is a constrained form of Gauss-Jordan elimination commonly used to compute persistent homology \cite{zomorodian_computing_2005, edelsbrunner_topological_2002}. 
The input to this algorithm is an $m \times m$ matrix $D$, and the output is a pair of matrices $R$ and $V$ such that $V$ is upper unitriangular, $R=DV$, and $R$ is reduced in the following sense. 

\begin{definition}
    Let $R$ be an $m \times m$ matrix, and define $\Col_j(R)$ as the $j$th column of $R$, and
    $$
    \low_j(R)
    = 
    \begin{cases}
        \max \{ i : R_{ij} \neq 0 \} & \Col_j(R) \neq 0 \\
        0 & \textrm{otherwise}.
    \end{cases}
    $$
    We call $R$ \emph{reduced} \cite{cohen2006vines} if $\low_i(R) \neq \low_j(R)$ whenever $i$ and $j$ are distinct nonzero columns of $R$.
    If $\Col_j(R)$ is nonzero then we call $R_{\low_j(R), j}$ the \emph{trailing coefficient} of column $j$.
\end{definition}

\begin{algorithm}[H]
 \caption{ Standard algorithm for persistent homology. }
\label{alg:standard}
\begin{algorithmic}[1]
\REQUIRE An $m \times m$ matrix $D$ with field coefficients
\ENSURE A reduced matrix $R$ and an upper unitrianuglar matrix $V$ such that $R = DV$
\STATE $R \gets D$
\STATE $V \gets I$, the $m \times m$ identity matrix
\FOR{$j = 1,\ldots,m$}
\STATE $i \gets \low_R(j) $
\WHILE{there exists $i > 0$ and $j' < j$ such that $\low_R(j') = i = \low_R(j)$}
\STATE $\Col_j(R) \gets \Col_j(R) - \frac{R_{ij}}{R_{ij'}} \Col_{j'}(R)$
\STATE $\Col_j(V) \gets \Col_j(V) - \frac{R_{ij}}{R_{ij'}} \Col_{j'}(V)$
\ENDWHILE
\ENDFOR
\RETURN $R, V$
\end{algorithmic}
\end{algorithm}

Now let $D$ be the boundary matrix for a simplex-wise filtration $\mathcal{K}$ on a simplicial complex $K = \{\sigma_1, \ldots, \sigma_{m}\}$ such that $\mathcal{K}_p = \{ \sigma_1, \ldots, \sigma_p\}$ for all $0 \le p \le m$. Assume the rows and columns of $D$ are arranged in sorted order, according to simplex, i.e. $\partial \sigma_j = \sum_i D_{ij} \sigma_i$.  Regard $D$ as a matrix with rational coefficients, and let $R=DV$ be the decomposition returned by the standard algorithm.

\begin{theorem}[Obayashi and Yoshiwaki \cite{obayashi_field_2023}]
    \label{thm:yoshiwakiAlgorithmCorrect}
 If $\mathcal{K}$ is a simplex-wise filtration on a simplicial complex then the persistent homology module $H_*(\mathcal{K}; \Z) = \bigoplus_d H_d(\mathcal{K}; \Z)$ splits as a direct sum of interval modules of free abelian groups if the trailing coefficient of every nonzero column of $R$ is a unit. In this case, one can obtain a cycle representative for each summand from the columns of $V$ and $R$.
\end{theorem}
\begin{proof}
    The proof is provided in \cite[p. 25, Condition 3]{obayashi2020field}; note that this is a preprint version of \cite{obayashi_field_2023}.
    The authors construct a decomposition explicitly by obtaining a basis for the space of chains $C(\mathcal{K}_m)$, denoted $\tilde \sigma_1, \ldots, \tilde \sigma_m$, which satisfies a certain compatibility criterion \cite[p. 25, Condition 3]{obayashi2020field} with the filtration and with the boundary operator $D$. One can then extract an appropriate cycle basis (c.f. Section \ref{sec:universal_cycle_reps}) from this larger basis in a manner consistent with \cite[Theorem 2.6]{de2011dualities}. Essentially, thanks to the compatibility conditions, each $\tilde \sigma_i$ either creates or destroys a persistent homology class, and the interval over which that homology class remains nontrivial can be determined by the birth time of the corresponding basis vector(s).
\qed \end{proof}

\begin{theorem}
    \label{thm:standardalgorithmcomplexity}
     The standard algorithm has worst-case $O(m^3)$ complexity.
\end{theorem}
\begin{proof}
    This result is well known (e.g., see \cite{obayashi_field_2023}) and simple to verify. If $z$ is the $j$th column of $D$ then we perform at most $j-1$ column additions to eliminate trailing entries in $z$, each addition consuming $O(m)$ algebraic operations. Summing over $j$ yields a total which is $O(m^3)$.
\qed \end{proof}

\subsection{General chain complexes}
\label{sec:generalChainComplexes}

We provide a new proof of Theorem \ref{thm:yoshiwakiAlgorithmCorrect}, and extend this result to a broader class of chain complexes. 

\begin{definition}
    A \emph{standard-form chain complex}\footnote{The term standard-form chain complex first appeared in \cite{hang2021umatch}.} is a chain complex $\Chains = \bigoplus_d \Chains_d$ such that $\Chains = \Z^{m \times 1}$ is the space of length-$m$ column vectors for some nonnegative integer $m$, and the following conditions hold.
    \begin{enumerate}
        \item For all $1 \le p \le m$, the $p$th standard unit vector, denoted  $\sigma_p = (0, \ldots, 1, \ldots, 0)$, is homogeneous. Concretely, this means that $\sigma_p \in \Chains_d$ for some $d$. This condition implies that each subgroup $\Chains_d$ is freely generated by a subset of the standard unit vectors.
        \item For each $0 \le p \le m$, the subspace $\mathcal{K}_p = \spann \{ \sigma_1, \ldots, \sigma_p\}$ is a sub-chain complex of $\Chains.$
    \end{enumerate}    
We call the filtration $\mathcal{K}=\{\mathcal{K}_p : 0 \le p \le m\}$ the \emph{canonical filtration} on $\Chains$.
\end{definition}

It is straightforward to transform any simplex-wise filtration on a simplicial complex into the canonical filtration on a standard-form chain complex, by representing the $p$th simplex in the filtration as the $p$th standard unit vector. Therefore this setting is slightly more general than the one addressed in Theorem \ref{thm:yoshiwakiAlgorithmCorrect}. 

Given a standard-form chain complex, we can regard the differential operator $D: \Chains \to \Chains$ as an $m \times m$ matrix with rational (in fact, integer) coefficients.
Let $R=DV$ be the decomposition produced by the standard algorithm, and let $J \in \mathbb{Q}^{m \times m}$ be the matrix such that 
    $$
    \Col_j(J) 
    =
    \begin{cases}
        \Col_r(R) \cdot \frac{1}{R_{\low_R(r),r}} & \low_r(R) = j \text{ for some (unique) positive index } r \\
        \Col_j(V) & \text{otherwise}\,.
    \end{cases}
    $$

\begin{lemma} {\; }
    \label{lem:umatch}
    \begin{enumerate}
        \item The matrix $J$ is upper unitriangular (upper triangular with diagonal entries $1$).
        \item The matrix  $M: = J^{-1} D J$ is a generalized matching matrix (i.e., each row and each column of $M$ contains at most one nonzero entry). 
        \item We have $M_{ij}=R_{ij}$ if $R_{ij}$ is the trailing coefficient of $\Col_j(R)$, and $M_{ij} =0$ otherwise.
        \item The matrices $V$, $R$, and $J$ have coefficients in $\Z$ if the trailing coefficient of every nonzero column of $R$ is a unit.
    \end{enumerate}
\end{lemma}

\begin{proof}
    Hang et al. \cite[Theorem 9, Assertion 2]{hang2021umatch}  show that the matrix $J$ fits into a U-match decomposition, given by $(J, M, D, J)$. That is, $J$ is upper unitriangular, $M$ is a generalized matching matrix, and $JM = DJ$ (hence, $M = J^{-1}DJ$). 
    They also show that $M_{ij} = R_{ij}$ if $R_{ij}$ is the trailing coefficient of $\Col_j(R)$, and $M_{ij} =0$ otherwise. 
    Now suppose that the trailing coefficient of every nonzero column of $R$ is a unit.  Then, by a simple induction,  each iteration of Lines 6 and 7 of Algorithm \ref{alg:standard} adds a vector of integers to a vector of integers.  Therefore $V$ and $R$ have integer coefficients. 
    Each column of $J$ is a column of either $V$ or $R$ (up to multiplication with a unit), so $J$ also has integer entries. 
    This completes the proof.
\qed \end{proof}

\begin{definition}
    Let $z$ be a cycle in $\Chains_k$, and let $p = \min\{ i : z \in \mathcal{K}_i\}$ be the minimum index such that $z$ is an element of $\mathcal{K}_i$. Then the \emph{principal submodule} of $H_k(\mathcal{K}; \Z)$ generated by $z$ is the smallest submodule $h \su H_k(\mathcal{K}; \Z)$ such that $[z] \in h_p$, where submodules are ordered with respect to inclusion.  
    We denote this submodule by 
    $$
    \Ps(z) = \bigcap \{ h \su H_k(\mathcal{K}; \Z): [z] \in h_p \}.
    $$
    Concretely, $\Ps(z)$ has form $0 \to \cdots \to 0 \to \langle z \rangle_p \to \cdots \to \langle z \rangle_m$, where  $ \langle z \rangle_q$ denotes the subgroup of $H_k(\mathcal{K}_q,\Z)$ generated by the homology class $[z]$.  Note that $\Ps(z)$ may be the zero module. Note, in addition, that $\Ps(z)$ is an interval module, provided that every cyclic group $\Ps(z)_q$ is either zero or free.
\end{definition}

\begin{theorem} 
\label{thm:StandardAlgorithmGeneralized}
The following are equivalent
\begin{enumerate}
    \item The trailing coefficient of every nonzero column of $R$ is a unit.
    \item The persistence module $H_k(\mathcal{K},\Z)$ splits as a direct sum of interval modules of free abelian groups, for all $k$. More concretely, 
    \begin{align}
    H_k(\mathcal{K},\Z) = \bigoplus_{ z \in E}  \Ps( z  ) 
    \label{eq:principalSubmoduleDirectSum}
    \end{align}
    where 
    $$
    E_ = \{ \Col_j(J) : 1 \le j \le m, \text{ and } \sigma_j \in \Chains_k, \text{ and } \Col_j(M)=0 \}\,.
    $$ 
\end{enumerate}
\end{theorem}
\begin{proof}   
Suppose that Condition 1 holds. Then $J$ is an upper unitriangular matrix with coefficients in $\Z$, by Lemma \ref{lem:umatch}. Because the diagonal entries of $J$ are units, it follows that $J^{-1} \in \Z^{m \times m}$. Thus $M=J^{-1} D J$ represents the boundary operator with respect to the basis of column vectors of $J$ over the integers.  This matrix is a generalized matching matrix, and its nonzero coefficients are the trailing entries of the nonzero columns of $R$,  by Lemma \ref{lem:umatch}. In particular, the nonzero coefficients of $M$ are units. By direct calculation, it follows that for all $k$ and $p$
    \begin{align*}
    H_k(\mathcal{K}_p,\Z) = \bigoplus \{ \langle z \rangle : z \in E_{k,p} \}
    \label{eq:pointwisebreakdown}
    \end{align*}
where $E_{k,p}$ is the set of all columns $z$ of $J$ such that $z$ is a $k$-cycle in $\mathcal{K}_p$. Note that many of these cycles may be zero, but this poses no difficulty, since $A = A \oplus 0$ for every abelian group $A$.
Equation \ref{eq:pointwisebreakdown} implies Equation \eqref{eq:principalSubmoduleDirectSum}, so we may conclude that Condition 1 implies Condition 2.

Now suppose that Condition 1 fails.  Then  $R_{\low_R(j),j} \notin \{-1,1\}$ for some $j$ such that $\Col_j(R) \neq 0$. If $ \sigma_j \in \Chains_{k+1}$, then it follows by direct calculation\footnote{{Consider the $j \times j$ submatrix in the upper-left-hand corner of $M$, which corresponds  to the sub-chain complex $\mathcal{K}_j$. Quotienting out the subcomplex $\mathcal{K}_{\low_R(j)-1}$ corresponds to deleting the first $\low_R(j)-1$ rows and columns of this matrix. The first row of the remaining matrix contains a single nonzero coefficient, namely $R_{\low_R(j),j}$. The cokernel of this matrix, regarded as a group homomorphism, will therefore contain a cyclic group of order $|R_{\low_R(j),j}|$.}} that $H_k(\mathcal{K}_j, \mathcal{K}_{\low_R(j)-1}; \Z)$ contains a cyclic subgroup of order $|R_{\low_R(j),j}|$. Thus $H_k(\mathcal{K}, \Z)$ does not split as a direct sum of interval modules, by Theorem \ref{thm:OurFieldIndependence}. Therefore Conditions 1 and 3 must fail, also. This completes the proof.
\qed \end{proof}

\subsection{Cohomology}
\label{sec:cohomology}

The crux of Theorem \ref{thm:StandardAlgorithmGeneralized} is obtaining an upper unitriangular matrix $J$ with integer coefficients such that $M = J^{-1}DJ$ is a generalized matching matrix.
As we saw in Section \ref{sec:generalChainComplexes}, we can obtain such a matrix by applying the standard algorithm (Algorithm \ref{alg:standard}) to $D$, regarded as a matrix with rational coefficients. 

This is attractive because the standard algorithm runs in time $O(m^3)$, which can be quite fast relative to the other methods we have considered in the present work. However, today's leading persistent homology solvers rarely apply the standard algorithm to $D$ directly. Rather, they apply it to the anti-transpose, $D^\perp$, i.e. the matrix obtained by transposing $D$ and reversing the order of rows and columns. At first blush, this appears to offer no obvious advantage: it is not clear how reducing $D^\perp$ might help one compute a persistence diagram, and the worst-case complexity of factoring $D$ versus $D^\perp$ is identical. 

However, the advantage of factoring $D^\perp$ instead of $D$ turns out to be substantial. A simple formula can be applied to extract the persistence diagram from \emph{either} factorization \cite{de2011dualities} and in the  majority of data sets where persistent homology is calculated, researchers find that factoring $D^\perp$ consumes orders of magnitude less time and memory than $D$ \cite{bauer2017phat}. This approach is commonly called \emph{computing cohomologically}, as it can be interpreted as a computation of persistent  cohomology \cite{de2011dualities}.

With this in mind, it is natural to ask if we can derive a result similar to Theorem \ref{thm:StandardAlgorithmGeneralized}, using a factorization of $D^\perp$ instead of $D$. This turns out to be the case. Let $C$ be a standard-form chain complex with integer coefficients, and let $D \in \Z^{m \times m}$ be the associated boundary operator. 

\begin{definition}
    Let $R$ be an $m\times m$ matrix. We define $\row_r(R)$ to be the $r$th row of $R$ and $\lead_A(r)$ to be the index of the leading entry in the $r$th row of a matrix $A$. Concretely, $\lead_A(r) = 0$ if $\row_r(A) = 0$ and $\lead_A(r) = \min \{ j : A_{rj} \neq 0\}$ otherwise.
\end{definition}

Apply the standard algorithm to the anti-transposed matrix $D^\perp$ to obtain a factorization $\Ralt = D^\perp \Valt$, and define a matrix $K \in \Q^{m \times m}$ as follows. 

        $$
        \row_i(K)
        =
        \begin{cases}
            \row_r(\Ralt^\perp) \cdot \frac{1}{ \Ralt_{r, \lead_{\Ralt^\perp}(r)} }& \lead_{\Ralt^\perp}(r) = i  \text{ for some (unique) positive index }r
            \\
            \row_i(\Valt^\perp) & \text{otherwise }
        \end{cases}
        $$

\begin{lemma} 
    \label{lem:umatchDUAL}
    \begin{enumerate}
        \item The matrix $K$ is upper unitriangular.
        \item The matrix  $M: = K D K^{-1}$ is a generalized matching matrix. 
        \item We have $M_{ij}=(\Ralt^\perp)_{ij}$ if $(\Ralt^\perp)_{ij}$ is the leading coefficient of $\row_i(\Ralt^\perp)$, and $M_{ij} =0$ otherwise.
        \item The matrices $\Valt$, $\Ralt$, and $K$ have coefficients in $\Z$ if the trailing coefficient of every nonzero column of $\Ralt$ is a unit.
    \end{enumerate}
\end{lemma}
\begin{proof}

    Hang et al. \cite[Theorem 9, Assertion 4]{hang2021umatch} show that there exists a U-match decomposition $(K^{-1}, M,D,K^{-1})$. That is, $K^{-1}$ is upper unitriangular (and hence, so is $K$), $M$ is a generalized matching matrix, and $K^{-1}M=DK^{-1}$ (hence, $M=KDK^{-1}$). They also show that $M_{ij} = (\Ralt^\perp)_{ij}$ if $(\Ralt^\perp)_{ij}$ is the leading coefficient of $\row_i(R)$, and $M_{ij} =0$ otherwise. 
    Now suppose that the trailing coefficient of every nonzero column of $R$ is a unit.  Then, by a simple induction,  each iteration of Lines 6 and 7 of Algorithm \ref{alg:standard} adds a vector of integers to a vector of integers.  
    Therefore $\Valt$ and $\Ralt$ have integer coefficients. 
    By definition, each row of $K$ is either a row of either $\Valt$ or $\Ralt$ (up to multiplication by a unit), so $K$ has integer entries.
    This completes the proof.
\qed \end{proof}

\begin{theorem} 
\label{thm:StandardAlgorithmGeneralizedCohomology}
The following are equivalent
\begin{enumerate}
    \item The trailing coefficient of every nonzero column of $\Ralt$ is a unit.
    \item The persistence module $H_k(\mathcal{K},\Z)$ splits as a direct sum of interval modules of free abelian groups, for all $k$. More concretely,
    \begin{align}
    H_k(\mathcal{K},\Z) = \bigoplus_{ z \in E}  \Ps( z  ) 
    \end{align}
    where 
    $$
    E = \{ \Col_j(K^{-1}) : 1 \le j \le m, \text{ and } \sigma_j \in \Chains_k, \text{ and } \Col_j(M)=0 \}\,.
    $$    
\end{enumerate}
\end{theorem}
\begin{proof} 
The proof is almost the same as the proof of Theorem \ref{thm:StandardAlgorithmGeneralized}. In fact, if we substitute $K^{-1}$ for $J$ and Lemma \ref{lem:umatchDUAL} for Lemma \ref{lem:umatch} in the proof of Theorem \ref{thm:StandardAlgorithmGeneralized}, then the argument carries through unchanged. The idea in both cases is to find a change of basis matrix with integer coefficients (whether that be $J$ or $K^{-1}$), which transforms $D$ to a generalized matching matrix $M$. Once $D$ is in this form, it is easy to calculate cycle representatives and relative homology groups.
\qed \end{proof}

\begin{rmk}
    In Section \ref{sec:generalChainComplexes} we defined $M$ as $J^{-1}DJ$ and in Section \ref{sec:cohomology} we defined $M$ as $KDK^{-1}$. These matrices turn out to be identical, because $(J, M, D, J)$ and $(K^{-1}, M, D, K^{-1})$ are both U-match decompositions of $D$, and the generalized matching matrices of  U-match decompositions are unique \cite[Theorem 2]{hang2021umatch}.
\end{rmk}

\subsection{Limitations}
\label{sec:obayashi_limitations}

A routine exercise shows that if $f$ is a (finitely-indexed) pointwise free and finitely-generated persistence module over $\Z$, then $f$ is isomorphic to $H_k(\mathcal{K}; \Z)$ for some filtration $\mathcal{K}$ on a simplicial complex $K$.  However, $\mathcal{K}$ may not be a  \emph{simplex-wise}, and only a limited subset of persistence modules are compatible with simplex-wise filtrations. For example, if $\mathcal{K}$ is simplex-wise then the ranks of $f_i$ and $f_{i+1}$ can differ by at most one for all $i$, because adding a simplex can only increase or decrease a Betti number by one. A similar limitation holds for standard-form chain complexes with integer coefficients, because adding a basis vector increases or decreases each Betti number by at most one. Moreover, the methods described above demonstrably fail when one loosens the constraint that $\mathcal{K}$ be a simplex-wise filtration (see Example \ref{ex:simplexwisefailure}).

{
\renewcommand{\theexample}{\ref{ex:simplexwisefailure}}
\begin{example}[Restated from Section \ref{sec:intro}]
    Let $C$ be a chain complex such that $C_0 = \Z$, $C_1 = \Z^2$, and $C_i = 0$ for $i \notin \{0,1\}$. Suppose that the boundary matrix $\partial: C_1 \to C_0$ is represented by the matrix $\begin{bmatrix}
        2 & 3
    \end{bmatrix}$, and let $\mathcal{K} = \{\mathcal{K}_0, \mathcal{K}_1\}$ be the filtration of chain complexes given by $\mathcal{K}_0 = 0$ and $\mathcal{K}_1 = C$. 
    Then the matrix $R$ produced by the standard algorithm will have a nonzero column with trailing coefficient equal to $2$, but the persistence module $H_k(\mathcal{K}; \Z)$ splits as a direct sum of interval modules of free abelian groups for all $k$.  
\end{example}
\addtocounter{example}{-1}
}


Consequently, the methods described in this section do not provide a fully general procedure to decompose arbitrary pointwise free and finitely-generated persistence module over $\Z$ into interval submodules.

\begin{acknowledgements}
We thank Vidit Nanda and Amit Patel for helpful references, and Mason A. Porter and Benjamin Spitz for helpful comments and suggestions. 
We also thank our two anonymous referees for their  constructive feedback. 
\end{acknowledgements}

\bibliographystyle{spmpsci}      
\bibliography{references}   


\end{document}